\documentclass[11pt]{article}

\usepackage[mac]{inputenc}

\usepackage[T1]{fontenc}
\usepackage{textcomp}
\usepackage{epstopdf}
\usepackage{epsfig}
\usepackage{graphicx}
\usepackage{hyperref}

\usepackage{subfigure}
\usepackage{url}

\usepackage{amsmath}
\usepackage{amsfonts}
\usepackage{amssymb}
\usepackage{pdfsync}
\usepackage{appendix}
\usepackage{dsfont}
\usepackage{tikz}
\usepackage{amsbsy}
\usepackage{amssymb}
\usepackage{amscd}
\usepackage{amsthm}

\newcommand{\thefont}[2]{\fontsize{#1}{#2}\fontshape{n}\selectfont}
\newcommand{\1}{\rlap{\thefont{10pt}{12pt}1}\kern.16em\rlap{\thefont{11pt}{13.2pt}1}\kern.4em}

\makeatletter
\def\singlespace{\def\baselinestretch{1}\@normalsize}

\title{{\sc Nonparametric adaptive time-dependent multivariate function estimation}}

\author{
  {\em J\'er\'emie Bigot}\footnote{Institut Sup\'erieur de l'A\'eronautique et de l'Espace,
D\'epartement Math\'ematiques, Informatique, Automatique,
10 Avenue \'Edouard-Belin, BP 54032-31055, Toulouse CEDEX 4, France. Email: jeremie.bigot@isae.fr}, \quad DMIA/ISAE - University of Toulouse \\ 
  {\em Theofanis Sapatinas}\footnote{Department of Mathematics and Statistics,
University of Cyprus, P.O. Box 20537, CY 1678, Nicosia, Cyprus. Email: fanis@ucy.ac.cy}, \quad University of Cyprus
  }

\newcommand{\bigO}[1]{\ensuremath{\mathop{}\mathopen{}O\mathopen{}\left(#1\right)}}
\newcommand{\smallO}[1]{\ensuremath{\mathop{}\mathopen{}o\mathopen{}\left(#1\right)}}

\newcommand{\EE}{\ensuremath{{\mathbb E}}}
\newcommand{\HH}{\ensuremath{\LL^{2}(\XX)}}

\newcommand{\LL}{\ensuremath{{\mathbb L}}}

\newcommand{\PP}{\ensuremath{{\mathbb P}}}
\newcommand{\RR}{\ensuremath{{\mathbb R}}}

\DeclareMathOperator*{\argmin}{arg\,min}

\newcommand{\var}{\mbox{Var}}

\newcommand{\bg}{\mbox{${\mathbf g}$}}
\newcommand{\bB}{\mbox{${\mathbf B}$}}

\newcommand{\balpha}{{\boldsymbol{\alpha}}}

\newcommand{\bfun}{\boldsymbol{f}}

\newcommand{\XX}{\ensuremath{{\mathcal X}}}
\newcommand{\YY}{\ensuremath{{\mathcal Y}}}
\newcommand{\risk}{\ensuremath{{\mathcal R}}}

\newtheorem{theo}{Theorem}

\newtheorem{lemma}{Lemma}
\newtheorem{definition}{Definition}

\numberwithin{equation}{section}

\begin{document}

\maketitle

\begin{abstract}
We consider the nonparametric estimation problem of time-dependent multivariate functions observed in a presence of additive cylindrical Gaussian white noise of a small intensity. We derive minimax lower bounds for the $L^2$-risk in the proposed
spatio-temporal model as the intensity goes to zero, when the underlying unknown response function is assumed to belong
to a ball of appropriately constructed inhomogeneous time-dependent multivariate functions, motivated by practical applications. Furthermore, we propose both non-adaptive linear and adaptive non-linear wavelet estimators that are asymptotically optimal (in the minimax sense) in a wide range of the so-constructed balls of inhomogeneous time-dependent multivariate functions. The usefulness of the suggested adaptive nonlinear wavelet estimator is illustrated with the help of simulated and real-data examples.
\end{abstract}

\noindent \emph{Keywords:} Adaptivity; Besov spaces; Block Thresholding;  Minimax Estimators; Time-dependent image processing, Wavelet Analysis.  

\noindent\emph{AMS classifications:} Primary 62G05, 62G08, 62G20; Secondary 62H35.

\section{Introduction}
\label{sec:intro}
The nonparametric estimation problem of high-dimensional objects has been considered in the literature over the last three decades. With the help of appropriate balls in function spaces, such as, H\"older, Sobolev or Besov balls, that measure smoothness of the unknown underlying high-dimensional object, 
asymptotical (as the sample size goes to infinity) optimal properties (in the minimax sense) of various linear and non-linear estimators, such as, kernel, spline or wavelet estimators, have been obtained (see, e.g., \cite{MR1045442}, \cite{MR1226450} (regression setting) and \cite{MR2640738} (density setting), and the references therein).

\medskip
These optimality properties were studied by \cite{MR1856685} in the case of time-dependent multivariate response functions.  By following a trend to derive theoretical properties, \cite{MR1856685} considered a ``continuous-time'' model for the estimation problem of time-dependent multivariate functions observed in a presence of 
additive cylindrical Gaussian white noise, that is, they considered
\begin{equation}
d Y_{\epsilon}(t,\underline{x}) = \bfun(t,\underline{x}) d \underline{x} + \epsilon d W(t,\underline{x}), \label{model:direct}
\end{equation}
where  $t \in T$ ($T$ is a compact subset of $\RR$) is the time variable, $\underline{x} \in \XX$ ($\XX$ is a compact subset of $\RR^{d}$, $d \geq 1$) is the space variable, $ \bfun \in  \LL^{2}(  T \times \XX )$ is the time-dependent multivariate function that we wish to estimate, $d W(t,\underline{x})$ is a cylindrical orthogonal Gaussian random measure (representing additive noise in the measurements), and $\epsilon >0$ is a small level of noise, that may let be going to zero for studying asymptotic properties. 

\medskip
A formal definition of a cylindrical orthogonal Gaussian random measure can be found in Section 2.1 of \cite{MR1856685}. Moreover, we understand \eqref{model:direct} in a generalized sense, that is, the observable elements are treated as linear functionals, so that the process $Y_{\epsilon}(t,\underline{x})$, $t \in T$, $\underline{x} \in \XX$, is correctly defined (see Section 
\ref{subsec:ssm}). Also, without loss of generality, in the sequel, we assume that $T=[0,1]$ and $\XX=[0,1]^d$.

\medskip
Assume periodic assumptions in each argument of $\bfun(t,\underline{x})$, $t \in T$, $\underline{x} \in \XX$. Consider H\"older continuity in $L^2(\XX)$ on the derivatives of $\bfun(t,\underline{x})$ with respect to $t \in T$, uniformly over $\underline{x} \in \XX$, and H\"older continuity in $L^2(T)$ on the partial derivatives of $\bfun(t,\underline{x})$ with respect to the elements of $\underline{x} \in \XX$, uniformly over $t \in T$. Then, under known a-priori smoothness (i.e., knowing the involved H\"older parameters) of $\bfun(t,\underline{x})$, \cite{MR1856685} constructed a non-adaptive kernel-projection (linear) estimator and obtained an asymptotical (as $\epsilon \rightarrow 0$) upper bound of its $L^2$-risk (on $\XX$), uniformly over a set $T_1 \subset T$, that depends on $\epsilon$ and the involved smoothness parameters (see, \cite{MR1856685}, Theorem 4.1). Moreover, they have showed that, asymptotically, this 
upper bound cannot be improved (see \cite{MR1856685}, Lemma 5.3), thus establishing the asymptotical optimality (in the minimax sense) of their suggested estimator.

\medskip
Our aim is twofold. From a theoretical point of view, we extend the asymptotical optimal convergence rates derived in \cite{MR1856685}. In particular, when smoothness is measured in appropriate balls of  inhomogeneous functions, constructed with the help of tensor-product wavelet bases and Besov spaces, with or without a-priori knowledge of the involved smoothness parameters, we construct, respectively, non-adaptive linear (projection) or adaptive non-linear (block-thresholding) wavelet estimators that achieve the established asymptotical optimal convergence rates under the $L^2$-risk. From a practical point of view, we demonstrate the usefulness of the suggested adaptive nonlinear wavelet thresholding estimator in practical applications. In particular, we show the superiority of the suggested estimator in terms of average mean squared error over pixel by pixel and slice by slice wavelet denoising estimators, both  with universal thresholds.

\medskip
The paper is organized as follows. Section \ref{sec:motex} provides a motivating example. Section \ref{sec:wb} contains a brief summary of the tensor-product wavelet bases and standard Besov spaces while Section \ref{sec:smoothness} discusses the function spaces that we consider to appropriately model the considered inhomogeneous time-dependent multivariate functions. Section \ref{lowerF} contains the minimax lower bounds for the $L^2$-risk. Section \ref{upperF} introduces both non-adaptive linear and adaptive non-linear wavelet estimators and provides their minimax upper bounds for the $L^2$-risk in a wide range of the so-constructed balls of inhomogeneous time-dependent multivariate functions. Section \ref{numerical} demonstrates the usefulness of the suggested adaptive nonlinear wavelet estimator with the help of simulated and real-data examples. Section \ref{conclusions} contains some concluding remarks. Finally, Section \ref{appF} (Appendix) provides two technical lemmas that are used in the proofs of the main theoretical results.

\section{A motivating example}
\label{sec:motex}

Increasingly, scientific studies yield time-dependent $d$-dimensional images, in which the observed data consist  of sets of curves recorded on the pixels of $d$-dimensional images observed at different times or wavelengths, see e.g.\ \cite{MR2649646}. Examples include temporal brain response intensities measured by functional magnetic resonance imaging (fMRI) \cite{Whitcher},  satellite remote sensing images of landscapes \cite{Ju}, and functional brain mapping using electroencephalography (EEG) and magnetoencephalography (MEG) \cite{Ou}. In many applications, the measured curves tend to be spiky and this requires flexible adaptive and local modeling of their variations.  The high dimensionality and noise that characterize such time-dependent images makes difficult the estimation of the evolution of each pixel intensity over time (or wavelength). 

We now discuss a specific application that motivates the estimation of $\bfun$ in  model \eqref{model:direct}, and the choice of the function spaces that we use to measure the smoothness of $\bfun$ (see Section \ref{sec:smoothness}). An example of application and data fitting into model  \eqref{model:direct} is satellite remote sensing imaging of landscapes, where the data are in the form of a multiband satellite 2-dimensional image of remote sensing measurements in various spectral bands of an area that contains roads, forests, vegetation, lakes and fields, see \cite{MR2649646}. As an illustrative example,  we display in Figure \ref{fig:ONERA}(a) a  typical temporal (or wavelength) slice (i.e., 2-dimensional grey-level image), and we plot in Figure \ref{fig:ONERA}(b) two curves (1-dimensional signals) corresponding to two selected pixels highlighted in blue and green in  Figure \ref{fig:ONERA}(a). With respect to model \eqref{model:direct}, Figure \ref{fig:ONERA}(a) corresponds to a noisy version of $\underline{x} \mapsto \bfun(t,\underline{x}) $ for some fixed $t \in T$, while   Figure \ref{fig:ONERA}(b) corresponds to a noisy version of $t \mapsto \bfun(t,\underline{x}) $ for some fixed $\underline{x} \in \XX$. \\
 

\begin{figure}[htbp]
\centering
\subfigure[]
{ \includegraphics[width=6cm]{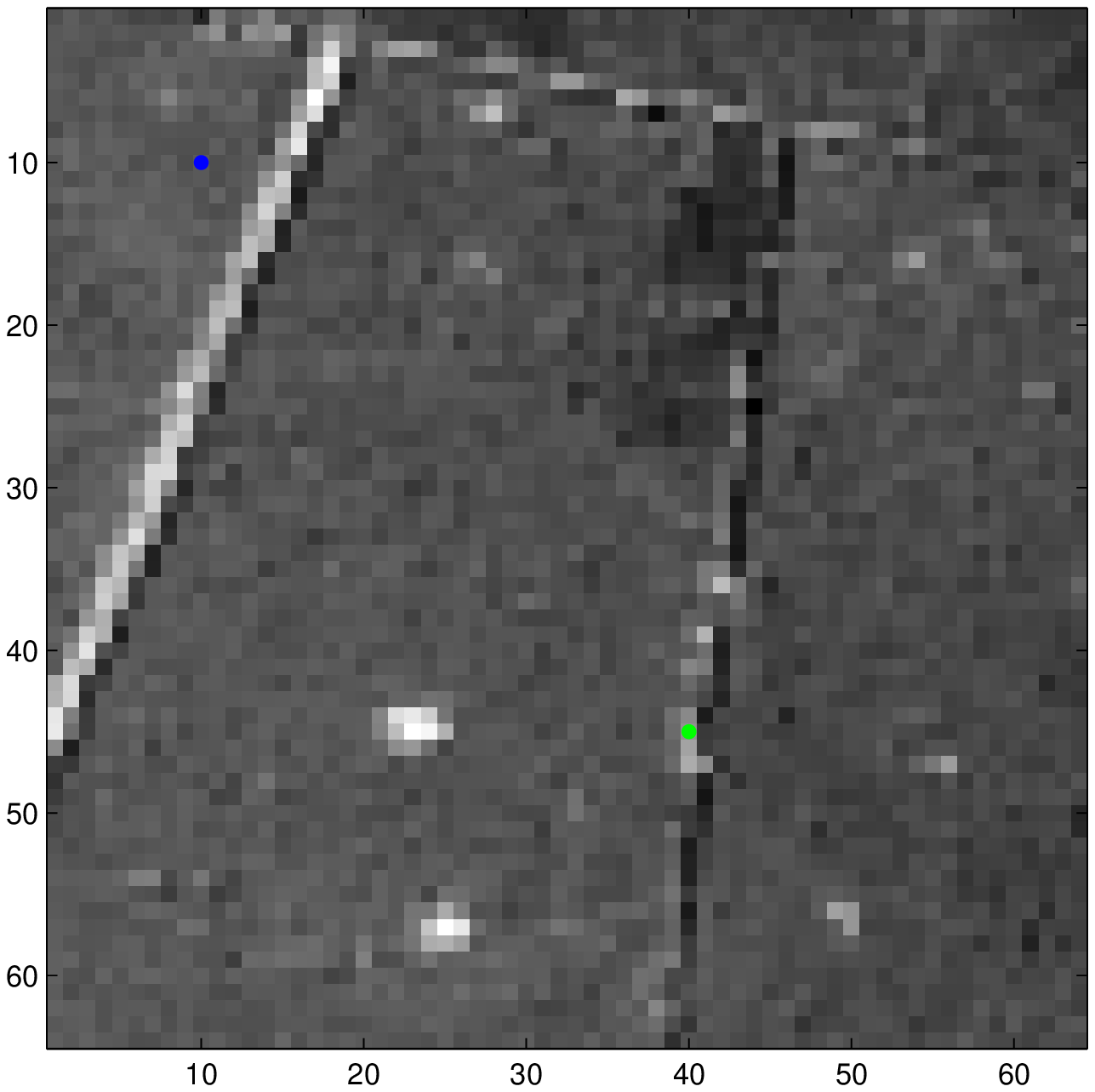}  }
\subfigure[]
{ \includegraphics[width=7.5cm]{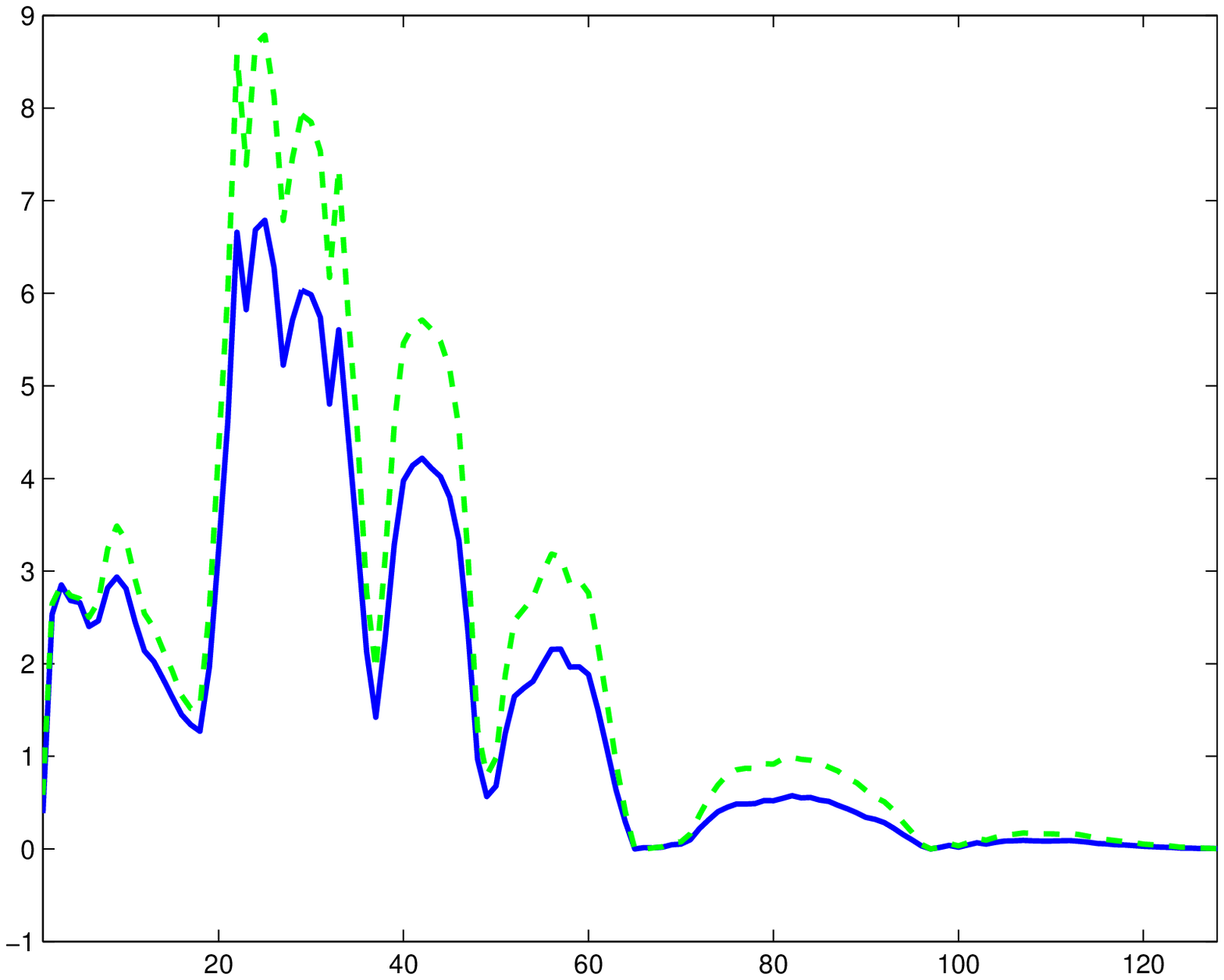} }

\caption{Satellite remote sensing image ($64 \times 64$ pixels, over $128$ wavelengths): (a) a  2-dimensional image measured at a specific wavelength; (b) evolution over wavelength of the intensities of the two pixels in green and blue marked in the image shown in (a).} \label{fig:ONERA}
\end{figure}

\section{Wavelets and Besov spaces}
\label{sec:wb}

We briefly consider tensor-product wavelet bases of  $\LL^{2}(\RR^d)$, $d \geq 1$, and recall some of their properties; for a detailed description of their construction, we refer to \cite{MR2479996}. Assume that we have at our disposal a 1-dimensional scaling function (i.e., a father wavelet) $\phi$ and a 1-dimensional wavelet function (i.e., a mother wavelet) $\psi$, both with compact supports. The scaling and wavelet functions of $\phi$ and $\psi$, at scale $j$ (i.e., at resolution level $2^{j}$) will be denoted by $\phi_{\lambda}$ and $\psi_{\lambda}$, respectively, where the index $\lambda$ summarizes both the usual scale and space parameters $j$ and $k$. In other words, for $d=1$, we set $\lambda = (j,k)$ and denote $\phi_{j,k} (\cdot)= 2^{j/2} \phi(2^{j}\cdot -k)$) and $\psi_{j,k}(\cdot) = 2^{j/2} \psi(2^{j}\cdot -k)$). For $d \geq 2$, the notation $\psi_{\lambda}$ stands for the adaptation of scaling and wavelet functions to $\RR^d$ (see  \cite{MR2479996}, Chapter 7). The notation $|\lambda| = j$ will be used to denote a wavelet at scale $j$, while $|\lambda| < j$ denotes a wavelet at scale $j'$, with $j_{0} \leq j' < j$,  where $j_{0}$ denotes the coarse level of approximation (usually called the primary resolution level). With the above notation, we assume that

\begin{description}
\item[-] the scaling functions $(\phi_{\lambda})_{|\lambda| = j}$ span a finite dimensional space $V_{j}$ within a multiresolution hierarchy $V_{0} \subset V_{1} \subset \ldots \subset L^{2}(\RR^d,dy)$, such that $\mbox{dim}(V_{j}) = 2^{jd}$.
\item[-] the scaling functions $(\phi_{\lambda})_{|\lambda| = j}$ form an orthonormal basis of $V_{j}$ and the wavelets $(\psi_{\lambda})_{|\lambda| = j}$ form an orthonormal basis of $W_{j}$ (with $W_j$ being  the orthogonal complement of $V_{j}$ into $V_{j+1}$).
\item[-]  Let $\YY$ be a compact subset of $\RR^d$, $d \geq 1$. Assuming periodicity in each argument of $\underline{y} \in \YY$,  and using standard wavelet bases ($d=1$) or tensor-product wavelet bases ($d \geq 2$)  of $L^2(\YY)$ (see, e.g. \cite{MR2479996}, Chapter 7),  any  $f \in \LL^{2}(\YY)$ can be decomposed as
$$
f(\underline{y})= \sum_{|\lambda| = j_{0}} c_{\lambda} \phi_{\lambda}(\underline{y}) + \sum_{j = j_{0}}^{+\infty} \sum_{|\lambda| = j} \beta_{\lambda} \psi_{\lambda}(\underline{y}), \quad \underline{y} \in \YY,
$$
where 
$$
c_{\lambda} = \langle f,\phi_{\lambda} \rangle_{\YY} \quad  \text{and} \quad \beta_{\lambda} = \langle f,\psi_{\lambda} \rangle_{\YY}.
$$
In order to simplify the notation, as it is commonly used, we write $(\psi_{\lambda})_{|\lambda| = j_{0}-1}$ for $(\phi_{\lambda})_{|\lambda| = j_{0}}$, and, thus, $f$ can be written in the compact form
$$
f(\underline{y}) =  \sum_{j = j_{0}-1}^{+\infty} \sum_{|\lambda| = j} \alpha_{\lambda} \psi_{\lambda}(\underline{y}), \quad \underline{y} \in \YY,
$$
where $\alpha_{\lambda}$ denotes either the scaling coefficients $c_{\lambda}$ or the wavelet coefficients $\beta_{\lambda}$.
\end{description}


\noindent
Consider also the following balls of (inhomogeneous) Besov spaces.

\medskip
Let $s_{1} > 0$ be a smoothness parameter in the domain $\YY_1$ (that is, $\YY$ with $d \geq 2$), and let $1 \leq p_{1},q_{1} \leq +\infty$. Let $(\psi_{\lambda})_{|\lambda| = j}$, $j \geq j_0$, be the (periodic) d-dimensional (tensor-product) compactly supported orthonormal wavelet basis of $L^{2}(\YY_1)$, with the convention that $(\psi_{\lambda})_{|\lambda| = j_{0}-1}$ denotes the scaling functions $(\phi_{\lambda})_{|\lambda| = j_{0}}$. Assume that the 1-dimensional scaling function $\phi$ and the 1-dimensional wavelet function $\psi$ are $\tau_{1}$-times continuously differentiable (regularity of the wavelet system ($\phi, \psi$)) with $0<s_{1}<\tau_{1}$, and assume that $s_{1} + d( 1/2 - 1/p_{1} ) > 0$. Define the norm $\| \cdot \|^{s_{1}}_{p_{1},q_{1}}$ by
$$\|f\|^{s_{1}}_{p_{1},q_{1}}  =\left( \sum_{j = j_0-1}^{+\infty} 2^{j(s_{1} + d( 1/2 - 1/p_{1})) q_{1}} \left( \sum_{ |\lambda|=j}
|\langle f,\psi_\lambda\rangle |^{p_{1}}\right)^{ q_{1}/p_{1}}\right) ^{1/q_{1}},$$
with the respective above sums replace by maximum if $p_{1}=+\infty$ and/or $q_{1} = +\infty$.  Then, the norm $\| \cdot \|^{s_{1}}_{p_{1},q_{1}}$ is equivalent to the traditional Besov norm (see e.g  \cite{MR1618204} for further details), and one can thus define the following Besov ball of radius $A_{1} > 0$
$$
B^{s_{1}}_{p_{1},q_{1}}(A_{1}) = \left\{f \in \LL^{2}(\YY_1), \;  \| f \|^{s_{1}}_{p_{1},q_{1}} \leq A_{1} \right\}.
$$

Let $s_{2} > 0$ be a smoothness parameter in the domain $\YY_2$ (that is $\YY$ with $d=1$), and let $1 \leq p_{2},q_{2} \leq +\infty$. Let $(\tilde{\psi}_{m,\ell})_{m=m_{0}-1,\ell=0,\ldots,2^{m}-1}$ be a (periodic) 1-dimensional compactly supported orthonormal wavelet basis of $\LL^{2}(\YY_2)$, with the convention that $(\tilde{\psi}_{m_{0}-1,\ell})_{\ell=0,\ldots,2^{m_{0}}-1}$ denotes the scaling functions $(\tilde{\phi}_{m_{0},\ell})_{\ell=0,\ldots,2^{m_{0}}-1}$, where $m_{0}$ is the coarse (primary) resolution level.  Assume that the corresponding 1-dimensional scaling function $\tilde{\phi}$ and the 1-dimensional wavelet function $\tilde{\psi}$ are $\tau_{2}$-times continuously differentiable (regularity of the wavelet system ($\tilde{\phi}, \tilde{\psi}$)) with $0<s_{2}<\tau_{2}$, and assume that $s_{2} + 1/2 - 1/p_{2}  > 0$. Define the norm $\| \cdot \|^{s_{2}}_{p_{2},q_{2}}$ by
$$
\| g\|^{s_{2}}_{p_{2},q_{2}}  =\left( \sum_{m = m_0-1}^{+\infty} 2^{m(s_{2} + 1/2 - 1/p_{2}) q_{2}} \left( \sum_{ m= 0}^{2^m-1} |\langle g, \tilde{\psi}_{m,\ell} \rangle |^{p_{2}}\right)^{ q_{2}/p_{2}}\right) ^{1/q_{2}},
$$
with the respective above sums replace by maximum if $p_{2}=+\infty$ and/or $q_{2} = +\infty$.  Then, as noticed above, one can define the following Besov ball of radius $A_{2} > 0$
$$
B^{s_{2}}_{p_{2},q_{2}}(A_{2}) = \left\{g \in \LL^{2}(\YY_2), \;  \| g \|^{s_{2}}_{p_{2},q_{2}} \leq A_{2} \right\}.
$$

\section{Smoothness assumptions on the time-dependent multivariate response function} 
\label{sec:smoothness}

The statistical problem that we consider below is the estimation of the unknown time-dependent multivariate response function $\bfun(t,\underline{x})$, $\underline{x} \in \XX$, $t \in T$,  based on observations from model \eqref{model:direct}. Motivated by the practical application discussed in Section \ref{sec:motex}, in order to derive the asymptotical (as $\epsilon \rightarrow 0$) optimal (in the minimax sense) rates of convergence (for the $L^2$-risk), we consider the following functional space to model $\bfun(t,\underline{x})$, $\underline{x} \in \XX$, $t \in T$.  
\medskip

First, let us assume that, for each $t \in T$, the mapping $\underline{x} \mapsto \bfun(t,\underline{x})$ belongs to $\LL^{2}(\XX)$. Let $\Lambda = \{ \lambda, |\lambda| = j\}_{j_{0}-1 \leq j \leq +\infty}$. For each $t \in T$, the (periodic) $d$-dimensional wavelet basis $(\psi_{\lambda})_{\lambda \in \Lambda}$  is used to decompose $\bfun(t,\underline{x})$ as 
\begin{equation}
\bfun(t,\underline{x}) = \sum_{j = j_{0}-1}^{+\infty} \sum_{|\lambda| = j} \balpha_{\lambda}(t) \psi_{\lambda}(\underline{x}) \quad \mbox{with} \quad \balpha_{\lambda}(t) = \langle \bfun(t,\cdot) , \psi_{\lambda} \rangle_{\HH}, \quad \underline{x} \in \XX. \label{eq:balpha}
\end{equation}
Then, for each $\lambda \in \Lambda$, we assume that the mapping  $t \mapsto 
\balpha_{\lambda}(t)$ belongs to $\LL^{2}(T)$. For each $\lambda \in \Lambda$, the (periodic) 1-dimensional wavelet basis $(\tilde{\psi}_{m,\ell})_{m=m_{0}-1,\ell=0,\ldots,2^{m}-1}$ is used to decompose $\balpha_{\lambda}(t)$ as
\begin{equation}
\balpha_{\lambda}(t) = \sum_{m=m_{0}-1}^{+\infty} \sum_{\ell = 0}^{2^{m}-1} \tilde{\alpha}_{\lambda,m,\ell} \tilde{\psi}_{m,\ell}(t) \quad \mbox{with} \quad \tilde{\alpha}_{\lambda,m,\ell} = \langle \balpha_{\lambda},\tilde{\psi}_{m,\ell}  \rangle_{\LL^{2}(T)}, \quad t \in T.  \label{eq:alpha}
\end{equation}
Finally, by assuming that the mapping $(t, \underline{x}) \mapsto \bfun(t,\underline{x})$ belongs to $\LL^{2}(T \times \XX)$ for any $t \in T$ and $x \in \XX$, and consider the corresponding tensor product wavelet basis, $\bfun(t,\underline{x})$ can thus be decomposed as
$$
\bfun(t,\underline{x}) =  \sum_{j = j_{0}-1}^{+\infty} \sum_{|\lambda| = j}  \sum_{m=m_{0}-1}^{+ \infty} \sum_{\ell = 0}^{2^{m}-1} \tilde{\alpha}_{\lambda,m,\ell} \tilde{\psi}_{m,\ell}(t) \psi_{\lambda}(\underline{x}), \quad t \in T, \quad \underline{x} \in \XX.
$$

\medskip
We are now ready to introduce the following definition in order to characterize the smoothness  of the time-dependent multivariate function $\bfun(t,\underline{x})$, $t \in T$, $x \in \XX$.

\begin{definition}  
Let $A_{1} >0$ and $A_{2} >0$ be constants. Let $s_{1} > 0$ be a smoothness parameter in space domain $\XX$ and $s_{2} > 0$ be a smoothness parameter in time domain $T$, such that $0 < s_{1} < \tau_{1}$ and  $0 < s_{2} < \tau_{2}$, where $\tau_1$ and $\tau_2$ are the regularity parameters of the wavelet systems $(\phi,\psi)$ and $(\tilde{\phi},\tilde{\psi})$, respectively. Let $1 \leq p_{1},q_{1} \leq +\infty$, $1 \leq p_{2},q_{2} \leq +\infty$, and assume that $s_{1} + d( 1/2 - 1/p ) > 0$ and  $s_{2} +  1/2 - 1/p_{2} > 0$. Let $p = (p_{1},p_{2})$ and $q = (q_{1},q_{2})$. Define $\bB_{p,q}^{s_{1},s_{2}}(A_{1},A_{2})$ as the following ball of functions in $\LL^{2}(  T  \times \XX )$:
$$
\bB_{p,q}^{s_{1},s_{2}}(A_{1},A_{2}) = \left\{ \bfun \in  \LL^{2}(T \times \XX)\;  |  \; \sup_{t \in T} \{ \| \bfun(t,\cdot) \|^{s_{1}}_{p_{1},q_{1}} \} \leq A_{1} \mbox{ and }\| \balpha_{\lambda} \|^{s_{2}}_{p_{2},q_{2}}  \leq A_{\lambda} \mbox{ for all }  \lambda \in \Lambda \right\},
$$
where, for each $t \in T$,
$$\|\bfun(t,\cdot)\|^{s_{1}}_{p_{1},q_{1}}  =\left( \sum_{j = j_0-1}^{+\infty} 2^{j(s_{1} + d( 1/2 - 1/p_{1})) q_{1}} \left( \sum_{ |\lambda|=j}
|\langle \bfun(t, \cdot),\psi_\lambda\rangle |^{p_{1}}\right)^{ q_{1}/p_{1}}\right) ^{1/q_{1}},
$$
for each $\lambda \in \Lambda$,
$$
\| \balpha_{\lambda} \|^{s_{2}}_{p_{2},q_{2}} =  \left( \sum_{m=m_{0}-1}^{+ \infty}  2^{m(s_{2} +  1/2 - 1/p_{2}) q_{2}} \left(\sum_{\ell = 0}^{2^{m}-1}   | \tilde{\alpha}_{\lambda,m,\ell} |^{p_{2}}\right)^{ q_{2}/p_{2}}\right) ^{1/q_{2}},
$$
with
$$
 \tilde{\alpha}_{\lambda,m,\ell} = \int_{T\times \XX}   \bfun(t,\underline{x})  \tilde{\psi}_{m,\ell}(t) \psi_{\lambda}(\underline{x})   dt d\underline{x},
$$
and $(A_{\lambda})_{\lambda \in \Lambda}$ is a a set of positive constants such that
\begin{equation}
 \sum_{j = j_{0}-1}^{+\infty} \sum_{|\lambda| = j}   A_{\lambda}^{2} \leq A_{2}^{2}. \label{Alambda}
\end{equation}

\end{definition}

\medskip
Assuming that $\bfun(t,\underline{x}) \in \bB_{p,q}^{s_{1},s_{2}}(A_{1},A_{2})$ means that the multivariate function $\bfun(t,\cdot)$ belong to $B^{s_{1}}_{p_{1},q_{1}}(A_{1})$, uniformly over $t \in T$. This assumption also means that the smoothness of the wavelet coefficients $(\balpha_{\lambda}(\cdot))_{\lambda \in \Lambda}$ over time $t \in T$ is measured by the parameter $s_{2}$ through a Besov ball $B^{s_{2}}_{p_{2},q_{2}}(A_{\lambda})$ whose radius satisfies equation \eqref{Alambda}. It implies that $\sup_{\lambda \in \Lambda} \left\{A_{\lambda}  \right\} \leq A_{2}$ and, more importantly, that $\lim_{j \to + \infty, \; |\lambda| = j}  A_{\lambda} = 0$ so that the Besov norm $ \| \balpha_{\lambda} \|^{s_{2}}_{p_{2},q_{2}}$ goes to zero as the resolution level of the time-dependent wavelet coefficients $\balpha_{\lambda} (\cdot)$ goes to infinity. In practical applications, it correspond to the assumption that the high-resolution energy of a time-dependent multivariate function, when integrated over time,  is going to zero.  (In order to simplify the notation, we have dropped the dependence of $\bB_{p,q}^{s_{1},s_{2}}(A_{1},A_{2})$ on $(A_{\lambda})_{\lambda \in \Lambda}$.)

\medskip

To motivate the definition of the functional space $\bB_{p,q}^{s_{1},s_{2}}(A_{1},A_{2})$, let us consider the real-data example on satellite remote sensing data discussed in Section \ref{sec:motex}. For this time-dependent 2-dimensional image, we display in Figure \ref{fig:wavcoeff} the curve $t \mapsto \balpha_{\lambda} (t)$ for two types of wavelet coefficients, one at a low resolution level ($|\lambda| = 3$) and  another one at the highest resolution level ($|\lambda| = 5$). Clearly, the curve at the highest resolution level has a smallest amplitude which is consistent with the decay of $A_{\lambda}$ as $|\lambda|$ increases in the definition of  $\bB_{p,q}^{s_{1},s_{2}}(A_{1},A_{2})$. Moreover, due the shape of the curves in Figure \ref{fig:wavcoeff}, it seems reasonable to assume that the functions $\balpha_{\lambda} (\cdot)$ have the same degree of smoothness $s_{2}$ across different resolution levels.

\begin{figure}[htbp]
\centering
\subfigure[]
{ \includegraphics[width=7.5cm]{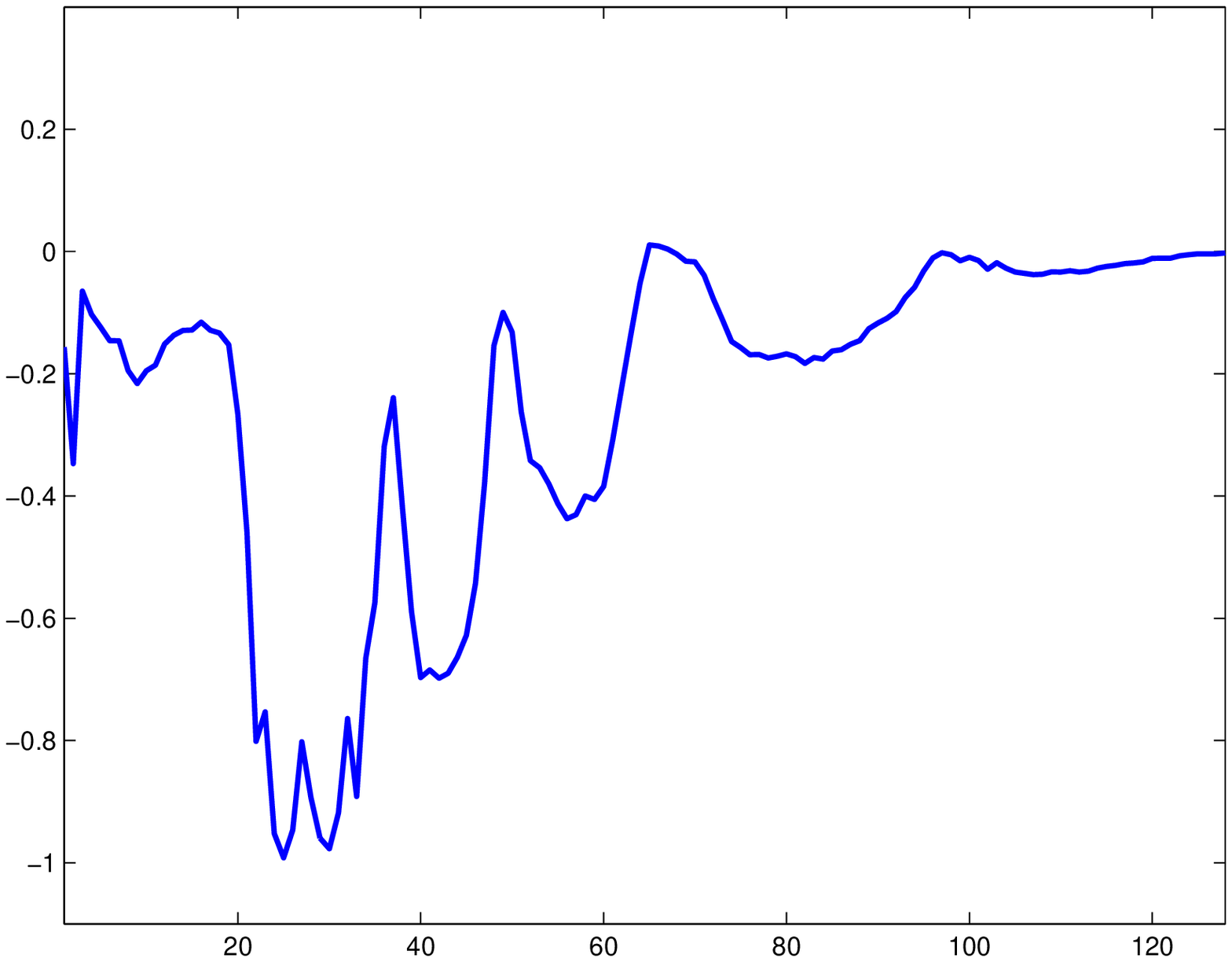} }
\subfigure[]
{ \includegraphics[width=7.5cm]{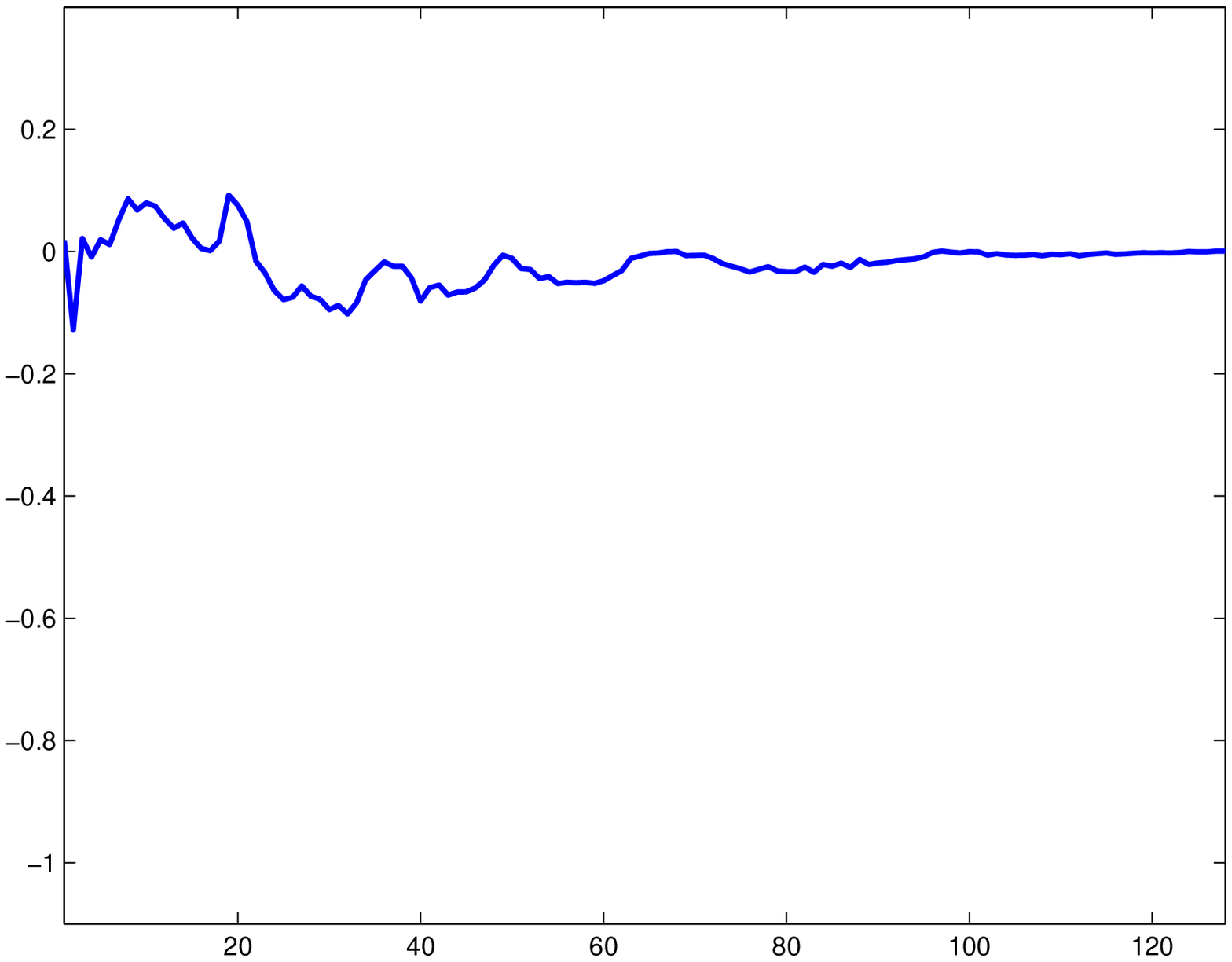} }

\caption{Satellite remote sensing image. Evolution of the curve  $t \mapsto \balpha_{\lambda} (t)$ for (a) a wavelet coefficient at resolution level $|\lambda| = 3$, and (b)  a wavelet coefficient at resolution level $|\lambda| = 5$.} \label{fig:wavcoeff}
\end{figure}

\medskip
In order to derive the minimax results, we define the minimax $L^2$-risk
over the class of balls $\bB_{p,q}^{s_{1},s_{2}}(A_{1},A_{2})$ as
\begin{eqnarray*}
\risk_{\epsilon}(\bB_{p,q}^{s_{1},s_{2}}(A_{1},A_{2})) &:=& \inf_{\hat{\bfun}_{\epsilon}} \sup_{\bfun \in  \bB_{p,q}^{s_{1},s_{2}}(A_{1},A_{2})} R(\hat{\bfun}_{\epsilon},\bfun),\\ &=& \inf_{\hat{\bfun}_{\epsilon}} \sup_{\bfun \in  \bB_{p,q}^{s_{1},s_{2}}(A_{1},A_{2})}  \EE\| \hat{\bfun}_{\epsilon} - \bfun\|^2 \\
&=& \inf_{\hat{\bfun}_{\epsilon}} \sup_{\bfun \in  \bB_{p,q}^{s_{1},s_{2}}(A_{1},A_{2})}   \bigg( \int_{T\times \XX}   \left|\hat{\bfun}_{\epsilon}(t,\underline{x}) - \bfun(t,\underline{x})\right|^2 dt d\underline{x} \bigg),
\end{eqnarray*}
where $\| {\bf g} \|$ is the $L^2$-norm of a function ${\bf g}$ defined on $T\times \XX$ and the
infimum is taken over all possible estimators $\hat{\bfun}_{\epsilon}$  
(i.e., measurable functions) of $\bfun$, based on observations from model \eqref{model:direct}.

\medskip
To present our results, for any $d \geq 1$ and any $s_{1} >0$ and $s_{2}>0$, we define $s>0$ to be such that
 \begin{equation} \label{def:s}
 \frac{1}{s} = \frac{1}{d+1}\left(\frac{d}{s_{1}} + \frac{1}{s_{2}}\right).
\end{equation}

\medskip
In what follows, we use the symbol $C$ for a generic positive
constant, independent of $\epsilon$, which may take different values at
different places. Moreover, in order to simplify the presentation of the results, and without loss of generality, we assume below that $j_{0} = m_{0} = 0$.

\section{Minimax lower bound for the $L^2$-risk}
\label{lowerF}

The following statement provides the minimax lower bounds for the $L^2$-risk.

\begin{theo} \label{theo:lowebound}
Let $A_{1} >0$ and $A_{2} >0$ be constants. Let $(A_{\lambda})_{\lambda \in \Lambda}$ be a set of positive constants satisfying (\ref{Alambda}), and assume that there exists a positive constant $A > 0$ such that, for any $-1 \leq j < + \infty$ and $ |\lambda| = j$,
\begin{equation}
A_{\lambda} 2^{\frac{d}{2}(j+1)} \geq A. \label{eq:condAlambda}
\end{equation}
Let $s_{1} > 0$  and $s_{2} > 0$ be the smoothness parameters in the space and time domains, respectively, such that $0 < s_{1} < \tau_{1}$ and  $0 < s_{2} < \tau_{2}$, where $\tau_1$ and $\tau_2$ are the regularity parameters of the wavelet systems $(\phi,\psi)$ and $(\tilde{\phi},\tilde{\psi})$, respectively. Assume that $1 \leq p_{1}, q_{1} \leq +\infty$, $1 \leq p_{2}, q_{2}  \leq +\infty$ such that $s_{1} + d( 1/2 - 1/p_{1} )  >0$ and  $s_{2} +  1/2 - 1/p_{2} >0$, and let $s > 0$ satisfy \eqref{def:s}. Then, there exists a constant $C > 0$ 
such that
$$
\risk_{\epsilon}(\bB_{p,q}^{s_{1},s_{2}}(A_{1},A_{2}))  \geq C\; \epsilon^{\frac{4s}{2s + d + 1}},
$$
for all sufficiently small $\epsilon > 0$.
\end{theo}

\begin{proof}

The proof is based on the standard Assouad's cube technique (see, e.g., \cite{MR2724359}, Chapter 2, Section 2.7.2).  Consider the following test functions 
$$
 \bfun_{w}(t,\underline{x}) = \mu_{j_{1},m_{2}} \sum_{j = -1}^{j_{1}}  \sum_{|\lambda| = j}  \sum_{m=-1}^{m_{2}} \sum_{\ell = 0}^{2^{m}-1}    w_{\lambda,m,\ell} \tilde{\psi}_{m,\ell}(t) \psi_{\lambda}(\underline{x}), \quad t \in T, \quad \underline{x} \in \XX,
$$
where $w = \left( (w_{\lambda,m,\ell})_{|\lambda| = j, 0 \leq \ell \leq  2^{m}-1 }  \right)_{j=-1,\ldots,j_{1},m=-1,\ldots, m_{2}} \in \Omega := \{1,-1\}^{2^{j_{1} d + m_{2}}}$,  and $\mu_{j_{1},m_{2}}$ is a positive sequence of reals satisfying the condition
\begin{equation}
 \mu_{j_{1},m_{2}}  =  c 2^{-\frac{1}{2}(m_{2}+1)} 2^{-\frac{d}{2}(j_{1}+1)}   \min\left(  2^{-(j_{1}+1) s_{1} }, 2^{-(m_{2}+1)s_{2} } \right),\label{eq:condmu}
\end{equation}
for some constant $c > 0$  not depending on $j_{1}$ and $m_{2}$. Assume that $c$ satisfies the condition
\begin{equation}
c \leq \min\left(A_{1} / (  \|\tilde{\psi}_{-1} \|_{\infty}  +  K \|\tilde{\psi} \|_{\infty} ), A \right), \label{eq:condc}
\end{equation} 
where $A$ is the constant satisfying inequality \eqref{eq:condAlambda} and $K$ is a constant that is proportional  to the length support of $\tilde{\psi}$. Then, it easily follows that  $\bfun_{w} \in \bB_{p,q}^{s_{1},s_{2}}(A_{1},A_{2})$ for any $w \in \Omega$. Indeed, for any $t \in T$, 
\begin{eqnarray*}
 \|  \bfun_{w}(t,\cdot) \|^{s_{1}}_{p_{1},q_{1}} & = & \left( \sum_{j =-1}^{j_{1}} 2^{j(s_{1} + d( 1/2 - 1/p_{1})) q_{1}} \left( \sum_{ |\lambda|=j} |\langle  \bfun_{w}(t,\cdot) ,\psi_\lambda\rangle |^{p_{1}}\right)^{ q_{1}/p_{1}}\right) ^{1/q_{1}},
\end{eqnarray*}
where
\begin{eqnarray*}
|\langle  \bfun_{w}(t,\cdot) ,\psi_\lambda\rangle | & \leq & \mu_{j_{1},m_{2}} \left( \|\tilde{\psi}_{-1} \|_{\infty} +  \sum_{m=0}^{m_{2}} \sum_{\ell = 0}^{2^{m}-1}  |\tilde{\psi}_{m,\ell}(t) |\right).
\end{eqnarray*}
Let us define the set
$$
I_{m}(t) = \{0 \leq \ell \leq 2^{m}-1 \; : \;  \tilde{\psi}_{m,\ell}(t) \neq 0 \}.
$$
Since, the wavelet $\tilde{\psi}$ is compactly supported, one has that the cardinality of $I_{m}(t)$ is bounded by a  constant $K > 0$ that is proportional  to the length support of $\tilde{\psi}$. Thus, using the relation $ \|\tilde{\psi}_{\lambda} \|_{\infty} = 2^{m/2}  \|\tilde{\psi} \|_{\infty} $, we obtain that, for any $t \in T$,
\begin{eqnarray*}
|\langle  \bfun_{w}(t,\cdot) ,\psi_\lambda\rangle | & \leq & \mu_{j_{1},m_{2}} \left( \|\tilde{\psi}_{-1} \|_{\infty} +  \sum_{m=0}^{m_{2}} \sum_{\ell \in I_{m}(t) }   |\tilde{\psi}_{m,\ell}(t) |\right) \\
 & \leq &  \mu_{j_{1},m_{2}} \left( \|\tilde{\psi}_{-1} \|_{\infty} +   \sum_{m=0}^{m_{2}}  K 2^{m/2}  \|\tilde{\psi} \|_{\infty} \right) \\
 & \leq &  \mu_{j_{1},m_{2}} \left( \|\tilde{\psi}_{-1} \|_{\infty} +   K   \|\tilde{\psi} \|_{\infty} 2^{\frac{1}{2}(m_{2}+1)} \right)  \\
 & \leq &  \mu_{j_{1},m_{2}} \left( \|\tilde{\psi}_{-1} \|_{\infty} +   K   \|\tilde{\psi} \|_{\infty} \right) 2^{\frac{1}{2}(m_{2}+1)}.
\end{eqnarray*}
Therefore, by the definition of  $\mu_{j_{1},m_{2}}$ given in (\ref{eq:condmu}), it follows that
\begin{eqnarray}
\sup_{t \in T} \{ \|  \bfun_{w}(t,\cdot) \|^{s_{1}}_{p_{1},q_{1}} \}   & \leq  & (  \|\tilde{\psi}_{-1} \|_{\infty}  + K  \|\tilde{\psi} \|_{\infty} ) \mu_{j_{1},m_{2}} 2^{\frac{1}{2}(m_{2}+1)} 2^{ (j_{1}+1) (s_{1} + d/2) }   \nonumber  \\
   & \leq  &  c  (  \|\tilde{\psi}_{-1} \|_{\infty}  +  K \|\tilde{\psi} \|_{\infty} ). \label{eq:condc1}
\end{eqnarray}
Now, define, for each $\lambda \in \Lambda$ and $t \in T$,
$$\balpha_{w,\lambda}(t)  =  \mu_{j_{1},m_{2}} \sum_{m=-1}^{m_{2}} \sum_{\ell = 0}^{2^{m}-1}    w_{\lambda,m,\ell} \tilde{\psi}_{m,\ell}(t),
$$
with
$w = \left( (w_{\lambda,m,\ell})_{|\lambda| = j, 0 \leq \ell \leq  2^{m}-1 }  \right)_{j=-1,\ldots,j_{1},m=-1,\ldots, m_{2}}$ and $\mu_{j_{1},m_{2}}$ are as given above.
On noting that 
$$
|\langle  \balpha_{w,\lambda} ,\tilde{\psi}_{m,l}\rangle | \leq \mu_{j_{1},m_{2}} |w_{\lambda,m,l}|,
$$
and using the  definition of  $\mu_{j_{1},m_{2}}$ given in (\ref{eq:condmu})
 and the inequality \eqref{eq:condAlambda}, we obtain that, for any $|\lambda| = j$ with $-1 \leq j \leq j_{1}$,
\begin{eqnarray}
\| \balpha_{w,\lambda} \|^{s_{2}}_{p_{2},q_{2}}  & \leq &  \mu_{j_{1},m_{2}}  2^{(m_{2}+1)(s_{2} +  1/2)} \nonumber \\
& \leq & c 2^{-\frac{d}{2}(j_{1}+1)} \leq c A_{\lambda}/A.  \label{eq:condc2}
\end{eqnarray}
Hence, using the inequalities \eqref{eq:condc1} and \eqref{eq:condc2}, it follows that the condition  \eqref{eq:condc}
is sufficient to imply that  $\bfun_{w} \in \bB_{p,q}^{s_{1},s_{2}}(A_{1},A_{2})$. 

\medskip 
In the rest of the proof, we will thus assume that condition \eqref{eq:condc} holds. 
Furthermore, we use the notation $\EE_{\bfun_{w}}$ to denote expectation with respect to the distribution $\PP_{\bfun_{w}}$ of the random process $Y$ in model \eqref{model:direct}  under the hypothesis that $\bfun = \bfun_{w}$.

\medskip
The minimax risk $\risk_{\epsilon}(\bB_{p,q}^{s_{1},s_{2}}(A_{1},A_{2}))$ can be bounded from below as follows  
\begin{eqnarray*}
\risk_{\epsilon}:=\risk_{\epsilon}(\bB_{p,q}^{s_{1},s_{2}}(A_{1},A_{2})) & \geq &  \inf_{\hat{\bfun}_{\epsilon}} \sup_{w \in  \Omega} R(\hat{\bfun}_{\epsilon},\bfun_{w}) \\
& \geq &  \inf_{\hat{\bfun}_{\epsilon}} \sup_{w \in  \Omega}  \sum_{j = -1}^{j_{1}} \sum_{|\lambda| = j}  \sum_{m=-1}^{m_{2}} \sum_{\ell = 0}^{2^{m}-1} \EE_{\bfun_{w}} \left| \hat{\alpha}_{\lambda,m,\ell}^{\epsilon}  -  \mu_{j_{1},m_{2}}   w_{\lambda,m,\ell} \right|^2
\end{eqnarray*}
where
$$
 \hat{\alpha}_{\lambda,m,\ell}^{\epsilon} = \int_{T \times \XX}   \hat{\bfun}_{\epsilon}(t,\underline{x})  \tilde{\psi}_{m,\ell}(t) \psi_{\lambda}(\underline{x})   dt d\underline{x}.
$$
Then, define
$$
 \hat{w}_{\lambda,m,\ell}^{\epsilon} := \argmin_{v \in \{-1,1\}} \left| \hat{\alpha}_{\lambda,m,\ell}^{\epsilon}  - \mu_{j_{1},m_{2}}  v \right|,
$$
and  remark that the triangular inequality and the definition of $ \hat{w}_{\lambda,m,\ell}^{\epsilon}$ imply that
$$
\mu_{j_{1},m_{2}} \left| \hat{w}_{\lambda,m,\ell}^{\epsilon} -    w_{\lambda,m,\ell} \right| \leq 2 \left| \hat{\alpha}_{\lambda,m,\ell}^{\epsilon}  -  \mu_{j_{1},m_{2}}   w_{\lambda,m,\ell} \right|,
$$
which yields
\begin{eqnarray*}
\risk_{\epsilon} & \geq &  \inf_{\hat{\bfun}_{\epsilon}} \sup_{w \in  \Omega}  \frac{\mu_{j_{1},m_{2}}^2}{4}   \sum_{j = -1}^{j_{1}} \sum_{|\lambda| = j}  \sum_{m=-1}^{m_{2}} \sum_{\ell = 0}^{2^{m}-1} \EE_{\bfun_{w}} \left|  \hat{w}_{\lambda,m,\ell}^{\epsilon}  -      w_{\lambda,m,\ell} \right|^2 \\
& \geq &  \inf_{\hat{\bfun}_{\epsilon}}    \frac{\mu_{j_{1},m_{2}}^2}{4}  \frac{1}{\# \Omega} \sum_{w \in  \Omega} \sum_{j = -1}^{j_{1}} \sum_{|\lambda| = j}  \sum_{m=-1}^{m_{2}} \sum_{\ell = 0}^{2^{m}-1}   \EE_{\bfun_{w}} \left|  \hat{w}_{\lambda,m,\ell}^{\epsilon}  -      w_{\lambda,m,\ell} \right|^2.
\end{eqnarray*}

Replacing the sums $\sum_{j = -1}^{j_{1}} \sum_{|\lambda| = j}  \sum_{m=-1}^{m_{2}} \sum_{\ell = 0}^{2^{m}-1}$ by $\sum_{\lambda,m,\ell}$ (to simplify the notation),  for any $\lambda,m,\ell$ and $w \in \Omega$, define the vector $w^{(\lambda,m,\ell)} \in \Omega$ having all its components equal to $w$ expect the $(\lambda,m,\ell)$-th element. Let $\# A$ denote the cardinality of a finite set $A$. Then
\begin{eqnarray*}
\risk_{\epsilon} & \geq &  \inf_{\hat{\bfun}_{\epsilon}}    \frac{\mu_{j_{1},m_{2}}^2}{4}  \frac{1}{\# \Omega} \sum_{\lambda,m,\ell}  \sum_{w \in  \Omega \; : \;  w_{\lambda,m,\ell}  = 1} \left(  \EE_{\bfun_{w}} \left|  \hat{w}_{\lambda,m,\ell}^{\epsilon}  -       w_{\lambda,m,\ell} \right|^2 +  \EE_{\bfun_{w^{(\lambda,m,\ell)}}} \left|  \hat{w}_{\lambda,m,\ell}^{\epsilon}  -      w_{\lambda,m,\ell}^{(\lambda,m,\ell)} \right|^2 \right)  \\
 & \geq &  \inf_{\hat{\bfun}_{\epsilon}}    \frac{\mu_{j_{1},m_{2}}^2}{4}  \frac{1}{\# \Omega} \sum_{\lambda,m,\ell}  \sum_{w \in  \Omega \; : \;  w_{\lambda,m,\ell}  = 1}  \EE_{\bfun_{w}} \left(  \left|  \hat{w}_{\lambda,m,\ell}^{\epsilon}  -       w_{\lambda,m,\ell} \right|^2 +   \left|  \hat{w}_{\lambda,m,\ell}^{\epsilon}  -       w_{\lambda,m,\ell}^{(\lambda,m,\ell)} \right|^2 \frac{d \PP_{\bfun_{w^{(\lambda,m,\ell)}}}}{d \PP_{\bfun_{w}}} (Y)  \right).
\end{eqnarray*}
Since $ w_{\lambda,m,\ell}^{(\lambda,m,\ell)} = - w_{\lambda,m,\ell} $ and $\hat{w}_{\lambda,m,\ell}^{\epsilon}  \in \{-1,1\}$, one finally obtains that, for any $0 < \delta < 1$,
\begin{eqnarray}
\risk_{\epsilon} & \geq &  \mu_{j_{1},m_{2}}^2   \frac{1}{\# \Omega} \sum_{\lambda,m,\ell}  \sum_{w \in  \Omega \; : \;  w_{\lambda,m,\ell}  = 1}  ^{2}  \;\EE_{\bfun_{w}} \left(  \min\left(1,   \frac{d \PP_{\bfun_{w^{(\lambda,m,\ell)}}}}{d \PP_{\bfun_{w}}} (Y) \right)   \right) \nonumber \\
 & \geq &  \delta \mu_{j_{1},m_{2}}^2   \frac{1}{\# \Omega} \sum_{\lambda,m,\ell}  \sum_{w \in  \Omega \; : \;  w_{\lambda,m,\ell}  = 1}  ^{2}\;  \PP_{\bfun_{w}} \left(   \frac{d \PP_{\bfun_{w^{(\lambda,m,\ell)}}}}{d \PP_{\bfun_{w}}} (Y) > \delta  \right). \label{eq:lowerdelta}
\end{eqnarray}
Thanks to the multiparameter Girsanov's formula, one has that, under the hypothesis that $\bfun = \bfun_{w}$  in model \eqref{model:direct}, 
$$
\log \left(  \frac{d \PP_{\bfun_{w^{(\lambda,m,\ell)}}}}{d \PP_{\bfun_{w}}} (Y) \right) = \epsilon^{-1} \int_{T\times \XX}   (\bfun_{w^{(\lambda,m,\ell)}} - \bfun_{w}   )  (t,\underline{x})   d W(t,\underline{x})  - \frac{\epsilon^{-2}}{2} \int_{T\times \XX}   (\bfun_{w^{(\lambda,m,\ell)}} - \bfun_{w}   )^2 (t,\underline{x})   dt d\underline{x}
$$
Therefore, the random variable 
$$
Z_{\lambda,m,\ell} := \log \left(  \frac{d \PP_{\bfun_{w^{(\lambda,m,\ell)}}}}{d \PP_{\bfun_{w}}} (Y) \right)
$$ 
is Gaussian with mean $\theta = -  \frac{\epsilon^{-2}}{2}  \mu_{j_{1},m_{2}}^2$ and variance $\sigma^2 =  \epsilon^{-2}   \mu_{j_{1},m_{2}}^2$, that do not  depend on $(\lambda,m,\ell)$.

\medskip
Now, let $s > 0$ satisfy  \eqref{def:s}.  Define $j_{1} = j_{1}(\epsilon)$ and $m_{2} = m_{2}(\epsilon)$ as
\begin{equation}
\label{eq:j1m2-lb}
2^{( j_{1}(\epsilon)+1)}=\lfloor \epsilon^{-\frac{2 s}{ (2s + d + 1)s_{1}}} \rfloor
\quad \mbox{and} \quad
2^{(m_{2}(\epsilon)+1)} =\lfloor \epsilon^{-\frac{2 s}{ (2s + d + 1)s_{2}}} \rfloor.
\end{equation}
Thanks to \eqref{eq:condmu}, it follows that there exists $c_{1} > 0$ such that
$$
\epsilon^{-2}   \mu_{j_{1}(\epsilon),m_{2}(\epsilon)}^2 \leq c_{1}
$$
for all sufficiently small $\epsilon > 0$. Hence, $Z_{\lambda,m,\ell} \sim N(\theta,\sigma^2)$ with $|\theta| \leq c_{1}/2$ and $\sigma^2 \leq c_{1}$  which implies that there exist  $\gamma > 0$ and $0 < \delta < 1$, that do not depend on  $(\lambda,m,\ell)$ and $\epsilon$, such that for all sufficiently small $\epsilon > 0$
$$
\PP_{\bfun_{w}} \left(   \frac{d \PP_{\bfun_{w^{(\lambda,m,\ell)}}}}{d \PP_{\bfun_{w}}} (Y)  > \log(\delta) \right) \geq \gamma.
$$
Hence, inserting the above inequality into \eqref{eq:lowerdelta}, it implies that
$$
\risk_{\epsilon} \geq  \frac{1}{2} \delta \gamma    \mu_{j_{1}(\epsilon),m_{2}(\epsilon)}^2 \;   2^{( j_{1}(\epsilon) +1)d + (m_{2}(\epsilon)+1)}.
$$
Using the expressions of $j_{1}(\epsilon)$ and $m_{2}(\epsilon)$ given in \eqref{eq:j1m2-lb}, together with \eqref{def:s} and \eqref{eq:condmu}, we finally obtain that there exists a constant $C > 0$, that does not depend on $\epsilon$, such that
$$
\risk_{\epsilon} \geq C \epsilon^{\frac{4 s}{ 2s + d + 1}},
$$
for all sufficiently small $\epsilon >0$, thus completing the proof of the theorem.

\end{proof}

\section{Minimax upper bound for the $L^2$-risk}
\label{upperF}

We now provide minimax upper bounds for the $L^2$-risk. This will accomplished by constructing appropriate estimators of $\bfun(t,\underline{x})$, $t \in T$, $\underline{x} \in \XX$, in the sequence space model.

\subsection{The sequence space model}
\label{subsec:ssm}

The suggested estimators in the following sections, will be constructed on the sequence space. Let us first recall that (see, e.g.,  \cite{MR1856685}) \eqref{model:direct} must be understood in the following sense: for any $\bg \in \LL^{2}(T\times \XX )$, 
$$
\int_{T\times \XX} \bg(t,\underline{x}) d Y(t,\underline{x}) = \int_{T\times \XX} \bg(t,\underline{x}) \bfun(t,\underline{x}) dt d\underline{x} + \epsilon \int_{T \times \XX} \bg(t,\underline{x}) d W(t,\underline{x})
$$
so that the integrand of ``the data'' $d Y(t,\underline{x})$ with respect to $ \bg(t,\underline{x})$ is a random variable that is normally distributed with mean 
$$
\EE  \left( \int_{T\times \XX} \bg(t,\underline{x}) d Y(t,\underline{x})  \right) = \int_{T\times \XX} \bg(t,\underline{x}) \bfun(t,\underline{x}) dt d\underline{x}
$$  
and variance 
$$
\var  \left( \int_{T\times \XX} \bg(t,\underline{x}) d Y(t,\underline{x})  \right) = \epsilon^{2} \int_{T \times \XX} |\bg(t,\underline{x})|^{2} dt d\underline{x}.
$$
Moreover, for any $g_{1}, g_{2} \in \LL^{2}(T\times \XX )$ 
$$
\EE  \left( \int_{T \times \XX} \bg_{1}(t,\underline{x}) d W(t,\underline{x})  \int_{T \times \XX} \bg_{2}(t,\underline{x}) d W(t,\underline{x}) \right)= \int_{T \times \XX} g_{1}(t,\underline{x}) g_{2}(t,\underline{x}) dt d\underline{x}.
$$

Hence, in view of the above and using the tensor product wavelet basis constructed in Section \ref{sec:wb}, noisy observations of the coefficients $\tilde{\alpha}_{\lambda,m,\ell}$ are thus obtained through the following sequence model 
\begin{eqnarray}
y_{\lambda,m,\ell} &=& \int_{T \times \XX}  \tilde{\psi}_{m,\ell}(t) \psi_{\lambda}(\underline{x}) d Y(t,\underline{x}) \\ \nonumber &=& \tilde{\alpha}_{\lambda,m,\ell} + \epsilon \;z_{\lambda,m,\ell}, \; \lambda \in \Lambda, \; m \geq -1, \; \ell = 0,1,\ldots,2^{m}-1, \label{seqspam}
\end{eqnarray}
where the $z_{\lambda,m,\ell}$'s are independent and identically distributed (i.i.d.) standard Gaussian random variables, i.e., Gaussian random variables with zero mean and variance 1.

\subsection{Linear and non-adaptive estimator}
Consider the sequence space model \eqref{seqspam}. Let $j_{1} >0$ and $m_{2} >0$ be integers (smoothing parameters). We consider the following  non-adaptive wavelet projection (linear) estimator of $\bfun(t,\underline{x})$, $t \in T$, $\underline{x} \in \XX$, that is
\begin{equation}
\label{linearestF}
\hat{\bfun}_{j_{1},m_{2}}^{l}(t,\underline{x}) = \sum_{j = -1}^{j_{1}} \sum_{|\lambda| = j}  \sum_{m=-1}^{m_{2}} \sum_{\ell = 0}^{2^{m}-1} y_{\lambda,m,\ell} \tilde{\psi}_{m,\ell}(t) \psi_{\lambda}(\underline{x}), \quad t \in T, \quad \underline{x} \in \XX.
\end{equation}
Define the $L^2$-risk of $\hat{\bfun}_{j_{1},m_{2}}^{l}$ as
\begin{eqnarray*}
R(\hat{\bfun}_{j_{1},m_{2}}^{l},\bfun) &=&\EE \| \hat{\bfun}_{j_{1},m_{2}}^{l}  - \bfun \|^{2}_{\LL^{2}(T\times \XX)} \\
&=& \EE \bigg( \int_{T\times \XX}   \left|\hat{\bfun}_{j_{1},m_{2}}^{l}(t,\underline{x})  - \bfun(t,\underline{x})\right|^2 dt d\underline{x} \bigg).
\end{eqnarray*}

The following statement provides the minimax upper bounds for the $L^2$-risk of the non-adaptive (linear) wavelet estimator $\hat{\bfun}_{j_{1},m_{2}}^{l}$ given in \eqref{linearestF}.

\begin{theo} \label{theo:lin}
Let $A_{1} >0$ and $A_{2} >0$ be constants. Let $s_{1} > 0$  and $s_{2} > 0$ be the smoothness parameters in the space and time domains, respectively, such that $0 < s_{1} < \tau_{1}$ and  $0 < s_{2} < \tau_{2}$, where $\tau_1$ and $\tau_2$ are the regularity parameters of the wavelet systems $(\phi,\psi)$ and $(\tilde{\phi},\tilde{\psi})$, respectively. Assume that $2 \leq p_{1}, q_{1} \leq +\infty$, $2 \leq p_{2}, q_{2}  \leq +\infty$ such that $s_{1} + d( 1/2 - 1/p_{1} ) > 0$ and  $s_{2} +  1/2 - 1/p_{2} > 0$, and let $s > 0$ satisfy \eqref{def:s}. 
Consider the linear estimator $\hat{\bfun}_{j_{1},m_{2}}^{l}$ given in \eqref{linearestF}, and define $j_{1} = j_{1}(\epsilon)$ and $m_{2} = m_{2}(\epsilon) $
such that
\begin{equation}
\label{eq:j1m1}
2^{(j_{1}(\epsilon)+1)} = \lfloor \epsilon^{-\frac{2 s}{ (2s + d + 1)s_{1}}} \rfloor \quad \mbox{and} \quad 2^{(m_{2}(\epsilon)+1)} =  \lfloor   \epsilon^{-\frac{2 s}{ (2s + d + 1)s_{2}}} \rfloor.
\end{equation} 
 Then, there exists a constant $C > 0$ 
 such that
$$
\sup_{\bfun \in \bB_{p,q}^{s_{1},s_{2}}(A_{1},A_{2})} R(\hat{\bfun}_{j_{1}(\epsilon),m_{2}(\epsilon)},\bfun) \leq C \; \epsilon^{\frac{4s}{2s + d + 1}},
$$
for all sufficiently small $\epsilon > 0$.
\end{theo}

\begin{proof}
Let us write the usual bias-variance decomposition of the $L^2$-risk as
$$
R(\hat{\bfun}_{j_{1},m_{2}}^{l},\bfun) = B(\hat{\bfun}_{j_{1},m_{2}}^{l},\bfun) + V(\hat{\bfun}_{j_{1},m_{2}}^{l}),
$$
with 
$$
B(\hat{\bfun}_{j_{1},m_{2}}^{l},\bfun) =  \| \EE \hat{\bfun}_{j_{1},m_{2}}^{l}  - \bfun \|^{2}  
\quad \mbox{and} \quad V(\hat{\bfun}_{j_{1},m_{2}}^{l}) = \EE \| \hat{\bfun}_{j_{1},m_{2}}^{l}  - \EE \hat{\bfun}_{j_{1},m_{2}}^{l} \|^{2}.
$$
Obviously,
\begin{eqnarray}
V(\hat{\bfun}_{j_{1},m_{2}}^{l}) &=& \epsilon^{2} \sum_{j = -1}^{j_{1}} \sum_{|\lambda| = j}  \sum_{m=-1}^{m_{2}} \sum_{\ell = 0}^{2^{m}-1}  \EE |z_{\lambda,m,\ell}|^2 \nonumber \\
&=& \epsilon^{2} 2^{ (j_{1} +1) d + m_{2} + 1}, \label{varlin} 
\end{eqnarray}
and
$$
B(\hat{\bfun}_{j_{1},m_{2}}^{l},\bfun) = B_{1}(\hat{\bfun}_{j_{1},m_{2}}^{l},\bfun)  + B_{2}(\hat{\bfun}_{j_{1},m_{2}}^{l},\bfun),
$$
where
\begin{eqnarray*}
B_{1}(\hat{\bfun}_{j_{1},m_{2}}^{l},\bfun) & = & \sum_{j = j_{1}+1}^{+\infty} \sum_{|\lambda| = j}    \sum_{m=-1}^{+ \infty} \sum_{\ell = 0}^{2^{m}-1} \left|\tilde{\alpha}_{\lambda,m,\ell}\right|^2 \\
& = &   \sum_{j = j_{1}+1}^{\infty} \sum_{|\lambda| = j}    \int_{T} \left| {\balpha}_{\lambda}(t) \right|^2 dt   \\
& = &   \int_{T} \sum_{j = j_{1}+1}^{\infty} \sum_{|\lambda| = j}  \left| {\balpha}_{\lambda}(t) \right|^2 dt \\
\end{eqnarray*}
and
\begin{eqnarray*}
B_{2}(\hat{\bfun}_{j_{1},m_{2}}^{l},\bfun) & = & \sum_{j = -1}^{j_{1}} \sum_{|\lambda| = j}  \sum_{m=m_{2}+1}^{+ \infty} \sum_{\ell = 0}^{2^{m}-1} \left|\tilde{\alpha}_{\lambda,m,\ell}\right|^2 \\
& \leq &  \sum_{j = -1}^{+ \infty} \sum_{|\lambda| = j}  \sum_{m=m_{2}+1}^{+ \infty} \sum_{\ell = 0}^{2^{m}-1} \left|\tilde{\alpha}_{\lambda,m,\ell}\right|^2. 
\end{eqnarray*}

By  Lemma \ref{lemma:approx}, there exists a constant $K_{2} > 0$ (only depending on $s_{2}$ and $p_{2}$) such that
$$
 \sum_{m=m_{2}+1}^{+ \infty} \sum_{\ell = 0}^{2^{m}-1} \left|\tilde{\alpha}_{\lambda,m,\ell}\right|^2 \leq K_{2} A_{\lambda}^2 2^{-2(m_{2}+1)s_{2}}. 
$$
Thus, using \eqref{Alambda}, it follows that
\begin{equation} \label{B1}
B_{2}(\hat{\bfun}_{j_{1},m_{2}}^{l},\bfun) \leq K_{2} A_{2}^2 2^{-2(m_{2}+1)s_{2}}.
\end{equation}

Moreover,  by  Lemma \ref{lemma:approx}, there exists a constant $K_{1} > 0$ (only depending on $s_{1}$ and $p_{1}$) such that  
$$
 \sum_{j = j_{1}+1}^{+\infty} \sum_{|\lambda| = j}  \left| {\balpha}_{\lambda}(t) \right|^2 \leq K_{1} A_{1}^2 2^{-2(j_{1}+1)s_{1}},
 $$
This implies that
\begin{equation} \label{B2}
B(\hat{\bfun}_{j_{1},m_{2}}^{l},\bfun) \leq K_{1} A_{1}^2 2^{-2(j_{1}+1)s_{1}} + K_{2} A_{2}^2 2^{-2(m_{2}+1)s_{2}}.
\end{equation}
Therefore, by combining \eqref{varlin}, \eqref{B1} and \eqref{B2}, we arrive at
$$
R(\hat{\bfun}_{j_{1}(\epsilon),m_{2}(\epsilon)},\bfun) \leq K_{1} A_{1}^2 2^{-2 (j_{1}(\epsilon) +1)s_{1}} + K_{2} A_{2}^2 2^{-2(m_{2}(\epsilon)+1)s_{2}} + \epsilon^{2} 2^{ (j_{1}(\epsilon)+1) d + (m_{2}(\epsilon) +1)}.
$$
By taking into account the expressions of $j_{1}(\epsilon)$ and $m_{2}(\epsilon)$ given in \eqref{eq:j1m1}, together with \eqref{def:s}, 
we finally obtain that there exists a constant $C > 0$, that does not depend on $\epsilon$, such that
$$
\sup_{\bfun \in \bB_{p,q}^{s_{1},s_{2}}(A_{1},A_{2})} R(\hat{\bfun}_{j_{1}(\epsilon),m_{2}(\epsilon)},\bfun) \leq C \; \epsilon^{\frac{4s}{2s + d + 1}},
$$
for all sufficiently small $\epsilon > 0$, thus completing the proof of the theorem.
\end{proof}

The choice of the resolution levels $j_{1}$ and $m_{2}$ depends on the unknown smoothness parameters $s_{1}$ and $s_{2}$ in the space and time domains, respectively. The linear estimator $\hat{\bfun}_{j_{1},m_{2}}^{l}$ defined in \eqref{linearestF} is thus called non-adaptive (with respect to $s_1$ and $s_2$) and is  of limited interest in practical applications. Moreover, the results of Theorem \ref{theo:lin} are only suited to model  $d$-dimensional functions $\bfun(t,\cdot)$ belonging to the space $B^{s_{1}}_{p_{1},q_{1}}(A_{1})$ with $2 \leq p_{1},q_{1} \leq +\infty$, uniformly over $t \in T$. However, such Besov spaces are not suited to model spatially inhomogeneous multivariate functions.  

\medskip
In the following section, we thus consider the problem of constructing an adaptive non-linear estimator that is optimal (in the minimax sense) 
over Besov balls $\bB_{p,q}^{s_{1},s_{2}}(A_{1},A_{2}) $ with $1 \leq p_{1},q_{1} \leq  +\infty$ and $1 \leq p_{2},q_{2} \leq +\infty$.

\subsection{Non-linear and adaptive estimation} 
\label{sec:adapt}
Consider the sequence space model \eqref{seqspam}. For each $\lambda \in \Lambda$, we divide the wavelet coefficients $\alpha_{\lambda,m,\ell}$ at each resolution level $-1 \leq m < +\infty$ into blocks of length $L_{\epsilon} = 1 + \lfloor \log(\epsilon^{-2})  \rfloor$. Let $A_{m}$ and $U_{mr}$ be the following sets of indices
\begin{eqnarray*}
A_{m} & = & \left\{ r \; | \; r=1,2,\ldots,\frac{2^{m}}{ L_{\epsilon} } \right\}, \\
U_{mr} & = & \left\{ \ell \; | \; \ell=0,1,\ldots,2^{m}-1; (r-1)  L_{\epsilon} \leq \ell \leq r  L_{\epsilon} - 1 \right\}.
\end{eqnarray*}
Now, we define
\begin{equation}
B_{\lambda,m,r} = \sum_{\ell \in U_{mr} } \alpha_{\lambda,m,\ell}^{2} \quad \mbox{and} \quad  \hat{B}_{\lambda,m,r} = \sum_{\ell \in U_{mr} } y_{\lambda,m,\ell}^{2}.
\end{equation}
We consider an adaptive wavelet block-thresholding (non-linear) estimator $\bfun(t,\underline{x})$, $t \in T$, $\underline{x} \in \XX$, that is
\begin{equation}
\label{nonlinearestF}
\hat{\bfun}_{j_{1},m_{2}}^{nl}(t,\underline{x}) = \sum_{j = -1}^{j_{1}} \sum_{|\lambda| = j}  \sum_{m=-1}^{m_{2}} \sum_{r \in A_{m}} \sum_{\ell \in U_{mr}} y_{\lambda,m,\ell} \1_{\left\{ \hat{B}_{\lambda,m,r}  \geq t_{\epsilon, \delta} \right\}} \tilde{\psi}_{m,\ell}(t) \psi_{\lambda}(\underline{x}),  \quad t \in T, \quad \underline{x} \in \XX,
\end{equation}
where $\1_{A}$ is the indicator function of the set $A$, and the resolution levels $j_{1}$ and $m_{2}$, and the threshold  $ t_{\epsilon, \delta} $, will be defined below.

\medskip
Define the $L^2$-risk of $\hat{\bfun}_{j_{1},m_{2}}^{nl}$ as
\begin{eqnarray*}
R(\hat{\bfun}_{j_{1},m_{2}}^{nl},\bfun) &=&\EE \| \hat{\bfun}_{j_{1},m_{2}}^{nl}  - \bfun \|^{2}_{\LL^{2}(T\times \XX)} \\
&=& \EE \bigg( \int_{T\times \XX}   \left|\hat{\bfun}_{j_{1},m_{2}}^{nl}(t,\underline{x})  - \bfun(t,\underline{x})\right|^2 dt d\underline{x} \bigg).
\end{eqnarray*}

The following statement provides the minimax upper bounds for the $L^2$-risk of the adaptive (non-linear) wavelet estimator $\hat{\bfun}_{j_{1},m_{2}}^{nl}$ given in \eqref{nonlinearestF}.

\begin{theo} \label{theo:adapt}
Let $A_{1} >0$ and $A_{2} >0$ be constants. Let $s_{1} > 0$  and $s_{2} > 0$ be the smoothness parameters in the space and time domains, respectively, such that $0 < s_{1} < \tau_{1}$ and  $0 < s_{2} < \tau_{2}$, where $\tau_1$ and $\tau_2$ are the regularity parameters of the wavelet systems $(\phi,\psi)$ and $(\tilde{\phi},\tilde{\psi})$, respectively.  Assume that $1 \leq p_{1}, q_{1} \leq +\infty$, $1 \leq p_{2}, q_{2}  \leq +\infty$ such that $s_{1} + d( 1/2 - 1/p_{1} )  >0$ and  $s_{2} +  1/2 - 1/p_{2} >0$ if $ 2 \leq p_{1}, q_{1} \leq + \infty$ and $ 2 \leq p_{2}, q_{2} \leq + \infty$, respectively, and  $s_{1} \geq  d/p_{1} $ and  $s_{2} \geq 1/p_{2} $ if $1 \leq p_{1}, q_{1} < 2$ and 
$1 \leq p_{2}, q_{2} < 2$, respectively.  Let also $s > 0$ satisfy \eqref{def:s}.
%
Consider the non-linear estimator $\hat{\bfun}_{j_{1},m_{2}}^{nl}$ given in \eqref{nonlinearestF}, and define $j_{1} = j_{1}(\epsilon)$ and $m_{2} = m_{2}(\epsilon)$ as 
\begin{equation}
\label{eq:j1m1-nl}
2^{(j_{1}(\epsilon)+1)} = \lfloor \epsilon^{-2} \rfloor \quad \mbox{and} \quad 2^{(m_{2}(\epsilon)+1)} =  \lfloor \epsilon^{-2}   \rfloor.
\end{equation}
Define the threshold
$$
t_{\epsilon,\delta} =   \delta^{2} \epsilon^{2} L_{\epsilon},
$$
for some $\delta > 2(2\sqrt{2} +1)$.  Then, there exists a constant $C > 0$ 
 such that
$$
\sup_{\bfun \in \bB_{p,q}^{s_{1},s_{2}}(A_{1},A_{2})} R(\hat{\bfun}_{j_{1}(\epsilon),m_{2}(\epsilon)}^{nl},\bfun) \leq C\; \epsilon^{\frac{4s}{2s + d + 1}},
$$
for all sufficiently small $\epsilon > 0$.
\end{theo}


\begin{proof}
From Parseval's equality, we can decompose the $L^2$-risk of $\hat{\bfun}_{j_{1},m_{2}}^{nl}$ as follows
$$
R(\hat{\bfun}_{j_{1},m_{2}}^{nl},\bfun) = B_{1} + B_{2} + R_{1} + R_{2},
$$
where
\begin{eqnarray*}
B_{1}  & = &  \sum_{j = j_{1}+1}^{+\infty} \sum_{|\lambda| = j}    \sum_{m=-1}^{+\infty} \sum_{\ell = 0}^{2^{m}-1} \left|\tilde{\alpha}_{\lambda,m,\ell}\right|^2 =   \int_{T} \sum_{j = j_{1}+1}^{\infty} \sum_{|\lambda| = j}  \left| {\balpha}_{\lambda}(t) \right|^2 dt,\\
B_{2} & = &   \sum_{j = -1}^{j_{1}} \sum_{|\lambda| = j}  \sum_{m=m_{2}+1}^{+\infty} \sum_{\ell = 0}^{2^{m}-1} \left|\tilde{\alpha}_{\lambda,m,\ell}\right|^2, \\
R_{1} & = & \sum_{j = -1}^{j_{1}} \sum_{|\lambda| = j}  \sum_{m=-1}^{m_{2}} \sum_{r \in A_{m}} \sum_{\ell \in U_{mr}} \EE \left( \left|y_{\lambda,m,\ell} - \tilde{\alpha}_{\lambda,m,\ell}\right|^2 \1_{\left\{ \hat{B}_{\lambda,m,r}  \geq t_{\epsilon, \delta} \right\}} \right),  \\
R_{2} & = & \sum_{j = -1}^{j_{1}} \sum_{|\lambda| = j}  \sum_{m=-1}^{m_{2}} \sum_{r \in A_{m}} \sum_{\ell \in U_{mr}} \EE \left( \left|\tilde{\alpha}_{\lambda,m,\ell}  \right|^2 \1_{\left\{ \hat{B}_{\lambda,m,r}  < t_{\epsilon, \delta} \right\}}  \right).
\end{eqnarray*}
To bound the risk, we need to control the terms $B_{1}$, $B_{2}$, $R_{1}$ and $R_{2}$. Let $p_{1}' = \min(p_{1},2)$ and $p_{2}' = \min(p_{2},2)$. Define also $s_{1}' = s_{1}  + d(1/2-1/p_{1}')$ and $s_{2}' = s_{2}  + 1/2-1/p_{2}'$. 

\medskip
By Lemma \ref{lemma:approx}, there exists a constant $K_{1}' > 0$, only depending on $s_{1}$ and $p_{1}$, such that 
$$
\sum_{j = j_{1}+1}^{+\infty} \sum_{|\lambda| = j}  \left| {\balpha}_{\lambda}(t) \right|^2 \leq K_{1}' A_{1}^2 2^{-2(j_{1}+1)s'_{1}}
$$
implying that
$$
B_{1} \leq K_{1}' A_{1}^2 2^{-2(j_{1}+1)s'_{1}}.
$$
Also, by  Lemma \ref{lemma:approx},  there exists a constant $K'_{2} > 0$, only depending on $s_{2}$ and $p_{2}$, such that
$$
 \sum_{m=m_{2}+1}^{+ \infty} \sum_{\ell = 0}^{2^{m}-1} \left|\tilde{\alpha}_{\lambda,m,\ell}\right|^2 \leq K'_{2} A_{\lambda}^2 2^{-2(m_{2}+1)s'_{2}},
$$
implying, in view of equation \eqref{Alambda}, that
$$
B_{2} \leq K_{2}' A_{2}^2 2^{-2(m_{2}+1)s'_{2}}.
$$
Consider the case $2 \leq p_{1} \leq +\infty$ implying that $s_{1}' =s_{1}$. Thanks to the definitions of $j_{1}(\epsilon)$ given in \eqref{eq:j1m1-nl} and $s$ given in \eqref{def:s}, we obtain
that
$$
B_{1} = \smallO{ \epsilon^{\frac{4s}{2s + d + 1}}}.
$$
In the case $1 \leq p_{1} < 2$, the condition $s_{1} \geq d/p_{1}$, the definitions of $j_{1}(\epsilon)$ given in \eqref{eq:j1m1-nl} and $s$ given  \eqref{def:s} also imply that
$$
B_{1} = \smallO{ \epsilon^{\frac{4s}{2s + d + 1}}}.
$$
Consider the case $2 \leq p_{2} \leq +\infty$ implying that $s_{2}' =s_{2}$. Thanks to the definitions of $m_{2}(\epsilon)$ given in \eqref{eq:j1m1-nl} and $s$ given in \eqref{def:s}, we obtain
that
$$
B_{2} = \smallO{ \epsilon^{\frac{4s}{2s + d + 1}}}.
$$
In the case $1 \leq p_{2} < 2$, the condition $s_{2} \geq 1/p_{2}$, the definitions of $m_{2}(\epsilon)$ given in \eqref{eq:j1m1-nl} and $s$ given in \eqref{def:s} also imply that
$$
B_{2} = \smallO{ \epsilon^{\frac{4s}{2s + d + 1}}}.
$$
Let us now write $R_{1}$ and $R_{2}$ as the sum of two terms
$$
R_{1} = R_{1,1} + R_{1,2} 
$$
where
\begin{eqnarray*}
R_{1,1} & = & \sum_{j = -1}^{j_{1}} \sum_{|\lambda| = j}  \sum_{m=-1}^{m_{2}} \sum_{r \in A_{m}} \sum_{\ell \in U_{mr}} \EE \left( \left|y_{\lambda,m,\ell} - \tilde{\alpha}_{\lambda,m,\ell}\right|^2 \1_{\left\{    \sum_{\ell \in U_{mr}}  \left|y_{\lambda,m,\ell} - \tilde{\alpha}_{\lambda,m,\ell}\right|^2   \geq \frac{1}{4} t_{\epsilon, \delta} \right\}} \right),  \\
R_{1,2} & = &  \sum_{j = -1}^{j_{1}} \sum_{|\lambda| = j}  \sum_{m=-1}^{m_{2}} \sum_{r \in A_{m}} \sum_{\ell \in U_{mr}} \EE \left( \left|y_{\lambda,m,\ell} - \tilde{\alpha}_{\lambda,m,\ell}\right|^2 \right) \1_{\left\{  B_{\lambda,m,r}  > \frac{1}{4} t_{\epsilon, \delta} \right\}} , \\
R_{2,1} & = & \sum_{j = -1}^{j_{1}} \sum_{|\lambda| = j}  \sum_{m=-1}^{m_{2}} \sum_{r \in A_{m}} \sum_{\ell \in U_{mr}} \EE \left( \left| \tilde{\alpha}_{\lambda,m,\ell}\right|^2 \1_{\left\{  \sum_{\ell \in U_{mr}}  \left|y_{\lambda,m,\ell} - \tilde{\alpha}_{\lambda,m,\ell}\right|^2   \geq  \frac{1}{4} t_{\epsilon, \delta} \right\}} \right),  \\
R_{2,2} & = &  \sum_{j = -1}^{j_{1}} \sum_{|\lambda| = j}  \sum_{m=-1}^{m_{2}} \sum_{r \in A_{m}} \sum_{\ell \in U_{mr}}  \left| \tilde{\alpha}_{\lambda,m,\ell}\right|^2  \1_{\left\{  B_{\lambda,m,r}  <  \frac{5}{2} t_{\epsilon, \delta} \right\}} ,
\end{eqnarray*}
where we have used the inequality  $\left|y_{\lambda,m,\ell} \right|^2 \leq 2 \left|y_{\lambda,m,\ell}  - \tilde{\alpha}_{\lambda,m,\ell}\right|^2 + 2 \left|\tilde{\alpha}_{\lambda,m,\ell}\right|^2$. \\

Let us first give an upper bound for $\Delta_{1} =  R_{1,1} + R_{2,1}$ as follows. Using Cauchy-Schwarz's inequality, moments properties of Gaussian random variables,  Lemma \ref{lemma:approx} and Lemma \ref{lem:deviation}, we have
\begin{eqnarray*}
\Delta_{1}   & \leq & \sum_{j = -1}^{j_{1}} \sum_{|\lambda| = j}  \sum_{m=-1}^{m_{2}} \sum_{r \in A_{m}} \sum_{\ell \in U_{mr}}  \left( \EE \left( \left|y_{\lambda,m,\ell} - \tilde{\alpha}_{\lambda,m,\ell}\right|^4 \right)^{1/2} + \left| \tilde{\alpha}_{\lambda,m,\ell}\right|^2  \right) \\
& & \times \left( \PP \left(\sum_{\ell \in U_{mr}}  \left|y_{\lambda,m,\ell} - \tilde{\alpha}_{\lambda,m,\ell}\right|^2   \geq \frac{1}{4} t_{\epsilon, \delta}\right)  
\right)^{1/2} \\
& \leq &  \left( \sqrt{3} \; 2^{(j_{1}+1)d + (m_{2}+1)}  \epsilon^{2} +  \sum_{j = -1}^{j_{1}} \sum_{|\lambda| = j}  \sum_{m=-1}^{m_{2}}  K'_{2} A_{\lambda}^{2}  2^{-ms_{2}'} \right) \epsilon^{\frac{1}{2}\left( \delta/2 -1\right)^2} \\
& = &\bigO{2^{(j_{1}+1)d + (m_{2}+1)}  \epsilon^{2 + \frac{1}{2}\left( \delta/2 -1\right)^2} + \epsilon^{\frac{1}{2}\left( \delta/2 -1\right)^2}   } \\
& = &\bigO{\epsilon^{-2}},
\end{eqnarray*}
where we have used the assumption that $\delta > 2(2\sqrt{2} +1)$.

Now, let $\Delta_{2} = R_{1,2} + R_{2,2}$. Let $j_{0}$ and $m_{0}$ be defined as
$$
2^{j_{0} +1} = \lfloor \epsilon^{-\frac{2s}{(2s+d+1)s_{1}'}}   \rfloor \quad \mbox{ and }  \quad 2^{m_{0} +1} =  \lfloor \epsilon^{-\frac{2s}{(2s+d+1)s_{2}'}}   \rfloor,
$$
where $s_{1}' = s_{1}  + d(1/2-1/p_{1}')$ and $s_{2}' = s_{2}  + 1/2-1/p_{2}'$. Note that $-1 \leq j_{0} < j_{1}$ and $-1 \leq m_{0} < m_{2}$ for all sufficiently small $\epsilon > 0$. Then, $\Delta_{2}$ can be partitioned as $\Delta_{2} = \Delta_{2,1} + \Delta_{2,2}$ where the first component $\Delta_{2,1}$ is calculated over the indices $-1 \leq j \leq j_{0}$ and  $-1 \leq m \leq m_{0}$, namely
$$
\Delta_{2,1} =  \sum_{j = -1}^{j_{0}} \sum_{|\lambda| = j}  \sum_{m=-1}^{m_{0}} \left[ \sum_{\ell =0}^{2^{m}-1}   \EE \left( \left|y_{\lambda,m,\ell} - \tilde{\alpha}_{\lambda,m,\ell}\right|^2 \right) \1_{\left\{  B_{\lambda,m,r}  > \frac{1}{4} t_{\epsilon, \delta} \right\}} +  \sum_{r \in A_{m}} B_{\lambda,m,r}    \1_{\left\{  B_{\lambda,m,r}  <  \frac{5}{2} t_{\epsilon, \delta} \right\}}    \right],
$$
and the second component $\Delta_{2,2}$ is calculated over the remaining indices, namely
\begin{eqnarray*}
\Delta_{2,2}& =& \sum_{j = j_{0}+1}^{j_{1}} \sum_{|\lambda| = j}  \sum_{m=-1}^{m_{2}} \left[ \sum_{\ell =0}^{2^{m}-1}   \EE \left( \left|y_{\lambda,m,\ell} - \tilde{\alpha}_{\lambda,m,\ell}\right|^2 \right) \1_{\left\{  B_{\lambda,m,r}  > \frac{1}{4} t_{\epsilon, \delta} \right\}} +  \sum_{r \in A_{m}} B_{\lambda,m,r}    \1_{\left\{  B_{\lambda,m,r}  <  \frac{5}{2} t_{\epsilon, \delta} \right\}}    \right] \\
& & +  \sum_{j = -1}^{j_{0}} \sum_{|\lambda| = j}  \sum_{m= m_{0}+1}^{m_{2}} \left[ \sum_{\ell =0}^{2^{m}-1}   \EE \left( \left|y_{\lambda,m,\ell} - \tilde{\alpha}_{\lambda,m,\ell}\right|^2 \right) \1_{\left\{  B_{\lambda,m,r}  > \frac{1}{4} t_{\epsilon, \delta} \right\}} +  \sum_{r \in A_{m}} B_{\lambda,m,r}    \1_{\left\{  B_{\lambda,m,r}  <  \frac{5}{2} t_{\epsilon, \delta} \right\}}    \right].
\end{eqnarray*}
Let us first give an upper bound for $\Delta_{2,1}$ as follows
\begin{eqnarray*}
\Delta_{2,1} & \leq &   \sum_{j = -1}^{j_{0}} \sum_{|\lambda| = j}  \sum_{m=-1}^{m_{0}}  \left[ \sum_{\ell =0}^{2^{m}-1}   \EE \left( \left|y_{\lambda,m,\ell} - \tilde{\alpha}_{\lambda,m,\ell}\right|^2 \right)  +  \sum_{r \in A_{m}}  \frac{5}{2} t_{\epsilon, \delta}   \right] \\
& = & \bigO{2^{(j_{0}+1)d + (m_{0}+1)} \epsilon^{2}} \\
& = & \bigO{\epsilon^{\frac{4s}{2s + d +1}} },
\end{eqnarray*}
where we have used the moments properties of Gaussian random variables,   and the fact that the blocks $A_{m}$ are of length $L_{\epsilon}$.\\

 Now, we compute an upper bound for $\Delta_{2,2}$.
  We have
 \begin{eqnarray*}
\Delta_{2,2}& \leq & \sum_{j = j_{0}+1}^{j_{1}} \sum_{|\lambda| = j}  \sum_{m=-1}^{m_{2}} \left[ \sum_{r \in A_{m}} \sum_{\ell \in U_{mr}}    \epsilon^{2}  \1_{\left\{  B_{\lambda,m,r}  > \frac{1}{4} t_{\epsilon, \delta} \right\}} +  \sum_{r \in A_{m}} B_{\lambda,m,r}     \right] \\
& & +  \sum_{j = -1}^{j_{0}} \sum_{|\lambda| = j}  \sum_{m= m_{0}+1}^{m_{2}} \left[ \sum_{r \in A_{m}} \sum_{\ell \in U_{mr}}    \epsilon^{2} \1_{\left\{  B_{\lambda,m,r}  > \frac{1}{4} t_{\epsilon, \delta} \right\}} +  \sum_{r \in A_{m}} B_{\lambda,m,r}     \right].
\end{eqnarray*}
Noticing that $\delta \geq 2$,  we see that $\sum_{r \in A_{m}} \sum_{\ell \in U_{mr}}    \epsilon^{2} \leq  \frac{1}{4} t_{\epsilon, \delta} $, which implies
 \begin{eqnarray*}
\Delta_{2,2}& \leq &2  \sum_{j = j_{0}+1}^{j_{1}} \sum_{|\lambda| = j}  \sum_{m=-1}^{m_{2}} \left[ \sum_{r \in A_{m}} B_{\lambda,m,r}      \right] \\
& & + 2 \sum_{j = -1}^{j_{0}} \sum_{|\lambda| = j}  \sum_{m= m_{0}+1}^{m_{2}} \left[    \sum_{r \in A_{m}} B_{\lambda,m,r}     \right].
\end{eqnarray*}
Then, noticing that $\sum_{r \in A_{m}} B_{\lambda,m,r}  =  \sum_{\ell = 0}^{2^{m}-1} \left|\tilde{\alpha}_{\lambda,m,\ell}\right|^2$,  by Lemma \ref{lemma:approx} we have
 \begin{eqnarray*}
\Delta_{2,2}& = & \bigO{  \sum_{j = j_{0}+1}^{j_{1}} \sum_{|\lambda| = j} \int_{T} \left| \balpha_{\lambda}(t) \right|^2dt  } +  \bigO{  \sum_{j = -1}^{j_{0}} \sum_{|\lambda| = j}   A_{\lambda}^2 \sum_{m= m_{0}+1}^{m_{2}}  \sum_{\ell = 0}^{2^{m}-1} \left|\tilde{\alpha}_{\lambda,m,\ell}\right|^2 } \\
& = & \bigO{  \sum_{j = j_{0}+1}^{j_{1}}  2^{-2 j s_{1}'}  }  +  \bigO{  \sum_{j = -1}^{j_{0}} \sum_{|\lambda| = j}  A_{\lambda}^2 \sum_{m= m_{0}+1}^{m_{2}}2^{-2 m s_{2}'}  } \\
& = & \bigO{  2^{-2 (j_{0}+1) s_{1}'}  }  +  \bigO{  2^{-2 (m_{0}+1) s_{2}'}  } = \bigO{\epsilon^{\frac{4s}{2s + d +1}} },
\end{eqnarray*}
using the definition of $j_{0}$ and $m_{0}$. This completes the proof of the theorem.
\end{proof}

\section{Numerical experiments}
\label{numerical}

We now illustrate the usefulness of the adaptive nonlinear wavelet estimator described in Section \ref{sec:adapt} with the help of simulated and real-data examples. The overall numerical study presented below has been carried out in the {\tt Matlab 7.7.0} programming environment.

\subsection{Simulated data}

We have used as a synthetic 2-dimensional (2D) example the Shepp-Logan phantom image (see \cite{Jain:1989:FDI:59921}) of size $N \times N$, with $N=64$ displayed in Figure \ref{fig:simus_setting}(a). This image is made of piecewise constant regions with different shape that partition the $N \times N$ pixels into 6 regions represented by different colors in  Figure \ref{fig:simus_setting}(a). To each pixel of a given region, we associate a one-dimensional (1D) signal of length $n = 128$. In this way, we are able to create a time-dependent 2D image $\left( \bfun(t_{\ell}, \underline{x}_{(k_{1},k_{2})}) \right)_{1 \leq \ell \leq n, 1 \leq k_{1}, k_{2} \leq N}$ that can be considered as the discretization of a function $\bfun : [0,1] \times [0,1]^{2} \to \RR$,  with $t_{\ell} = \frac{\ell}{n}$ and $\underline{x}_{(k_{1},k_{2})} = \left(\frac{k_{1}}{n},\frac{k_{2}}{n}\right)$.
\begin{figure}[htbp]
\centering
\subfigure[]
{ \includegraphics[width=5.8cm]{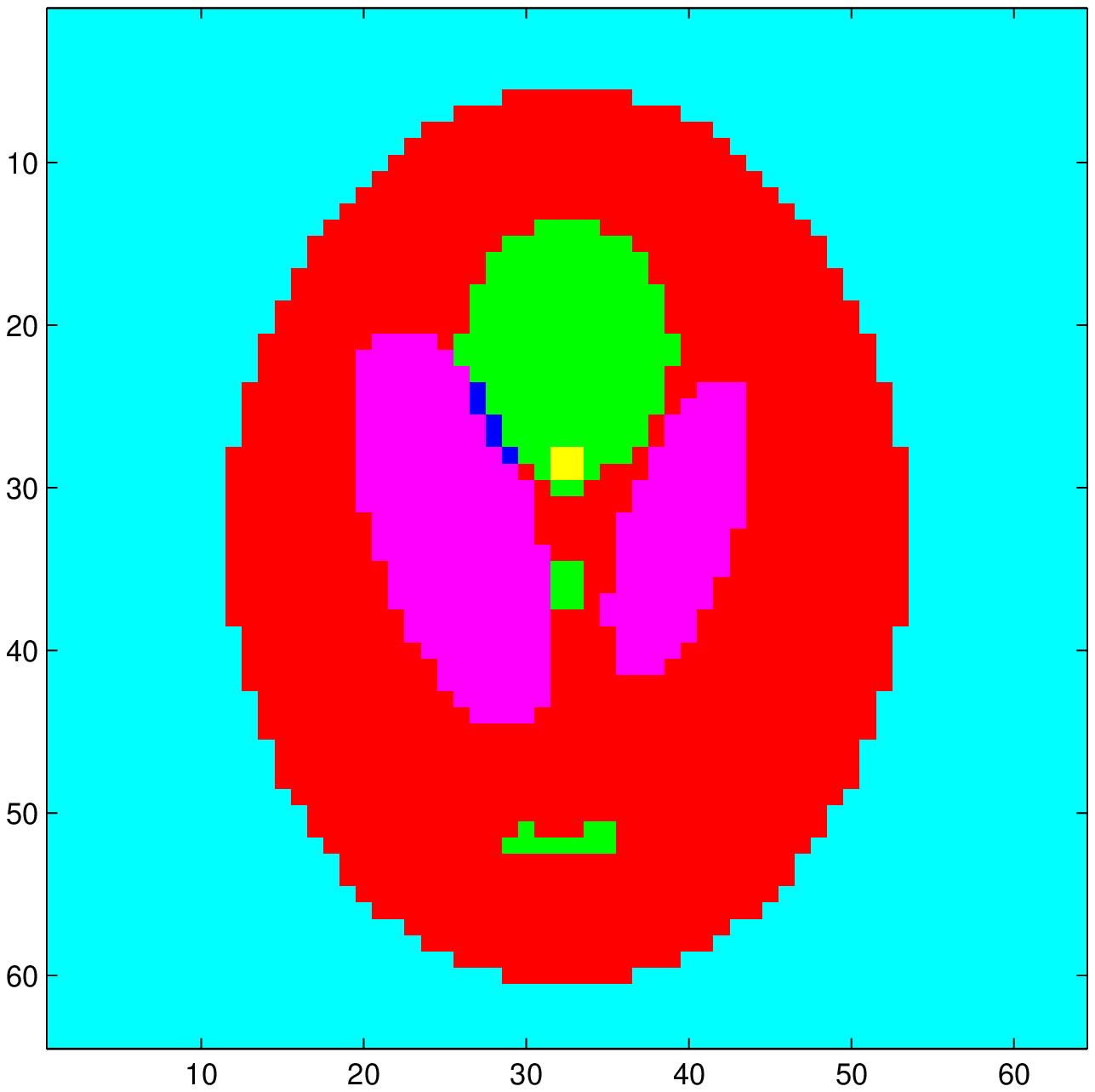}  }
\subfigure[]
{ \includegraphics[width=7.5cm]{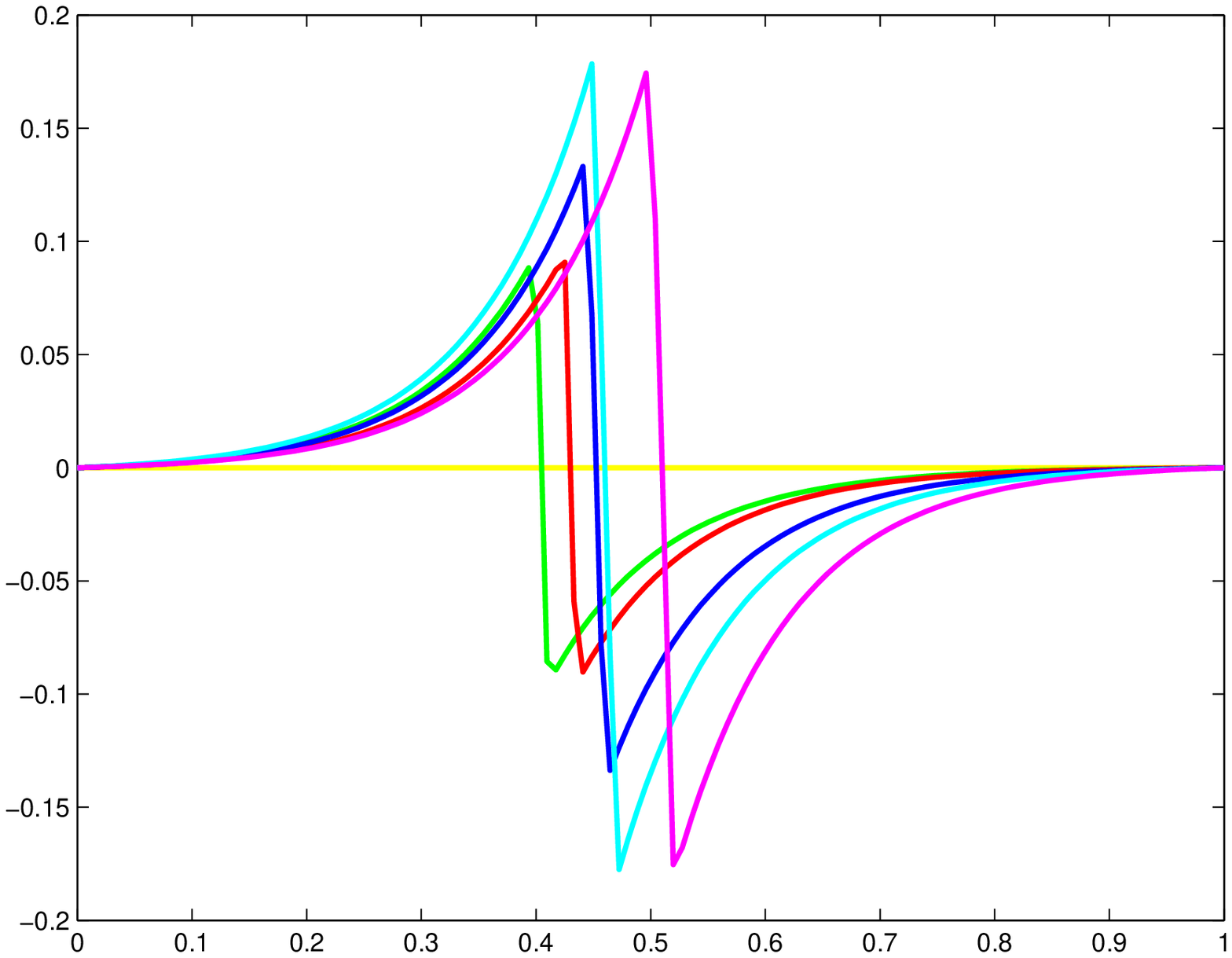} }

\caption{Synthetic example: (a) Shepp-Logan phantom made of 6 differents regions; (b) 1D signals associated to each region of the phantom image.} \label{fig:simus_setting}
\end{figure}

Then, we have created noisy data from the model
\begin{equation}
Y_{\ell,(k_{1},k_{2})} =  \bfun(t_{\ell}, \underline{x}_{(k_{1},k_{2})}) + \sigma w_{\ell,(k_{1},k_{2})}, \; 1 \leq \ell \leq n,  \; 1 \leq k_{1}, k_{2} \leq N, \label{eq:modsimus}
\end{equation}
where the $w_{\ell,(k_{1},k_{2})}$'s are i.i.d.\ standard Gaussian random variables, and $\sigma^2 > 0$ is the variance in the measurements ranging from a low to a high level in the simulations (we took signal-to-noise ratios equal to 7, 5 and 3). It is well known in nonparametric statistics (see e.g.\ \cite{MR1425958}) that there exists an asymptotic equivalence (in Le Cam sense) between the regression model \eqref{eq:modsimus} on $n N^2$ equi-spaced points, for each fixed $t \in T$,  and the white noise model \eqref{model:direct}, when taking $\epsilon = \frac{\sigma}{\sqrt{n N^2}}$. Therefore, thanks to this asymptotic equivalence, one can use the 2D+time dependent wavelet block thresholding approach described in Section \ref{sec:adapt} to denoise data from model \eqref{eq:modsimus}. To show the benefits of our approach, we compare it to two other mehods:
\begin{description}
\item[-] pixel by pixel denoising based on 1D wavelet thresholding: for each fixed pixel $(k_{1},k_{2})$, we apply a standard 1D wavelet-based denoising procedure with the universal threshold to the 1D data $\left(Y_{\ell,(k_{1},k_{2})}\right)_{1 \leq \ell \leq n}$,
\item[-]  slice by slice denoising based on 2D wavelet thresholding: for each fixed time $t_{\ell}$, we apply a standard 2D wavelet-based denoising procedure with the universal threshold to the 2D data $\left(Y_{\ell,(k_{1},k_{2})}\right)_{1 \leq k_{1}, k_{2} \leq N}$.
\end{description}
Then, we have generated $M=100$ repetitions of model \eqref{eq:modsimus} for three different values of the considered signal-to-noise ratio. For each replication, the quality of an estimate $\hat{\bfun}$ obtained by one of the above described methods is measured via its empirical mean squared error  
\begin{equation} \label{eq:MSE}
\frac{1}{n N^2} \sum_{\ell = 1}^{n} \sum_{k_{1},k_{2}=1}^{N} \left( \hat{\bfun}(t_{\ell}, \underline{x}_{(k_{1},k_{2})}) - \bfun(t_{\ell}, \underline{x}_{(k_{1},k_{2})})  \right)^2.
\end{equation}
The results of these simulations are displayed in Figure \ref{eq:modsimus} in the form of boxplots of the empirical mean squared error. Clearly, our approach yields the best results. The benefits of our method can also be clearly seen from the images displayed in Figure \ref{fig:simus_denoising} which show temporal cuts of the various estimators for a given simulation of the model.

\begin{figure}[htbp]
\centering
\subfigure[]
{ \includegraphics[width=5cm,height=5cm]{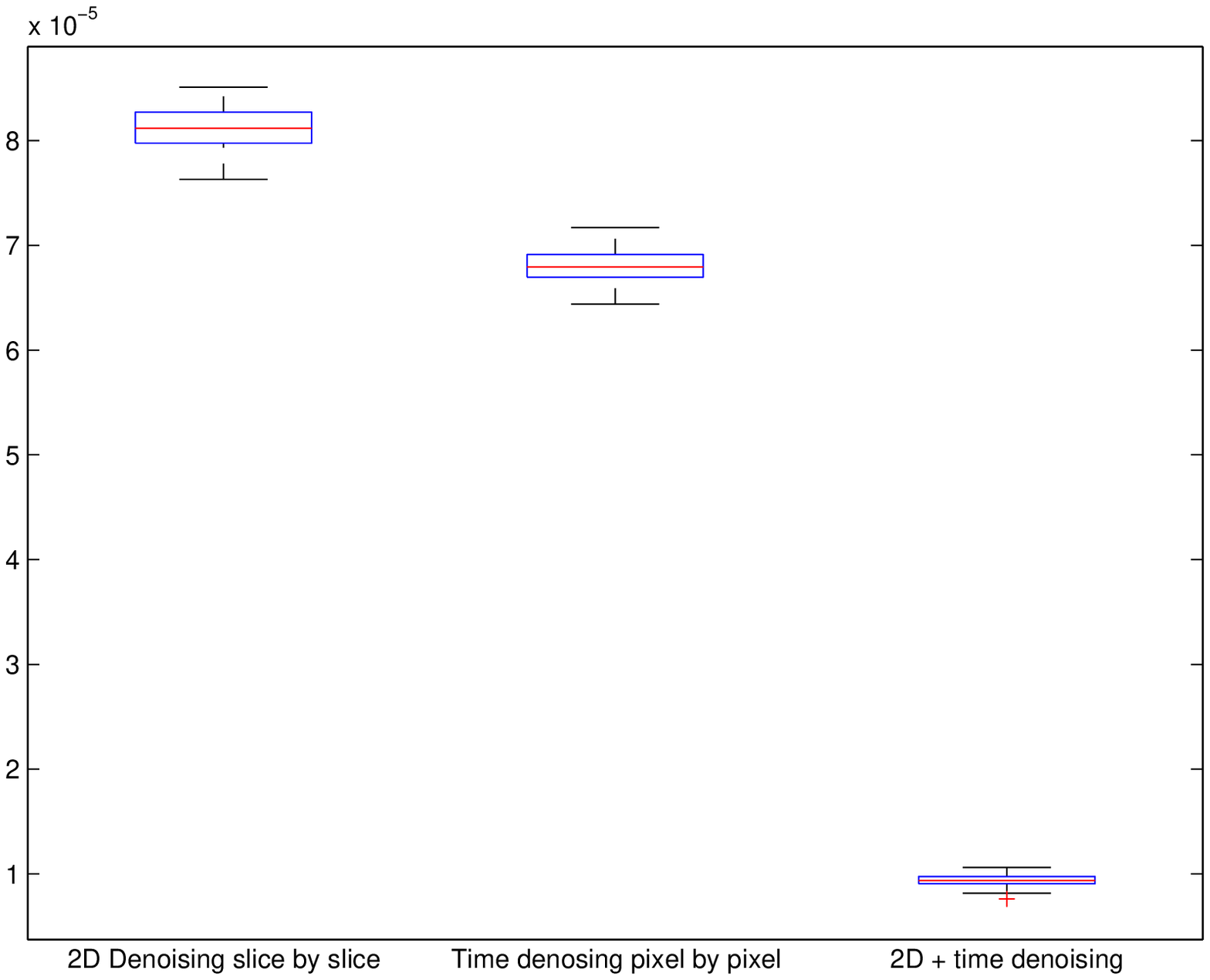}  }
\subfigure[]
{ \includegraphics[width=5cm,height=5cm]{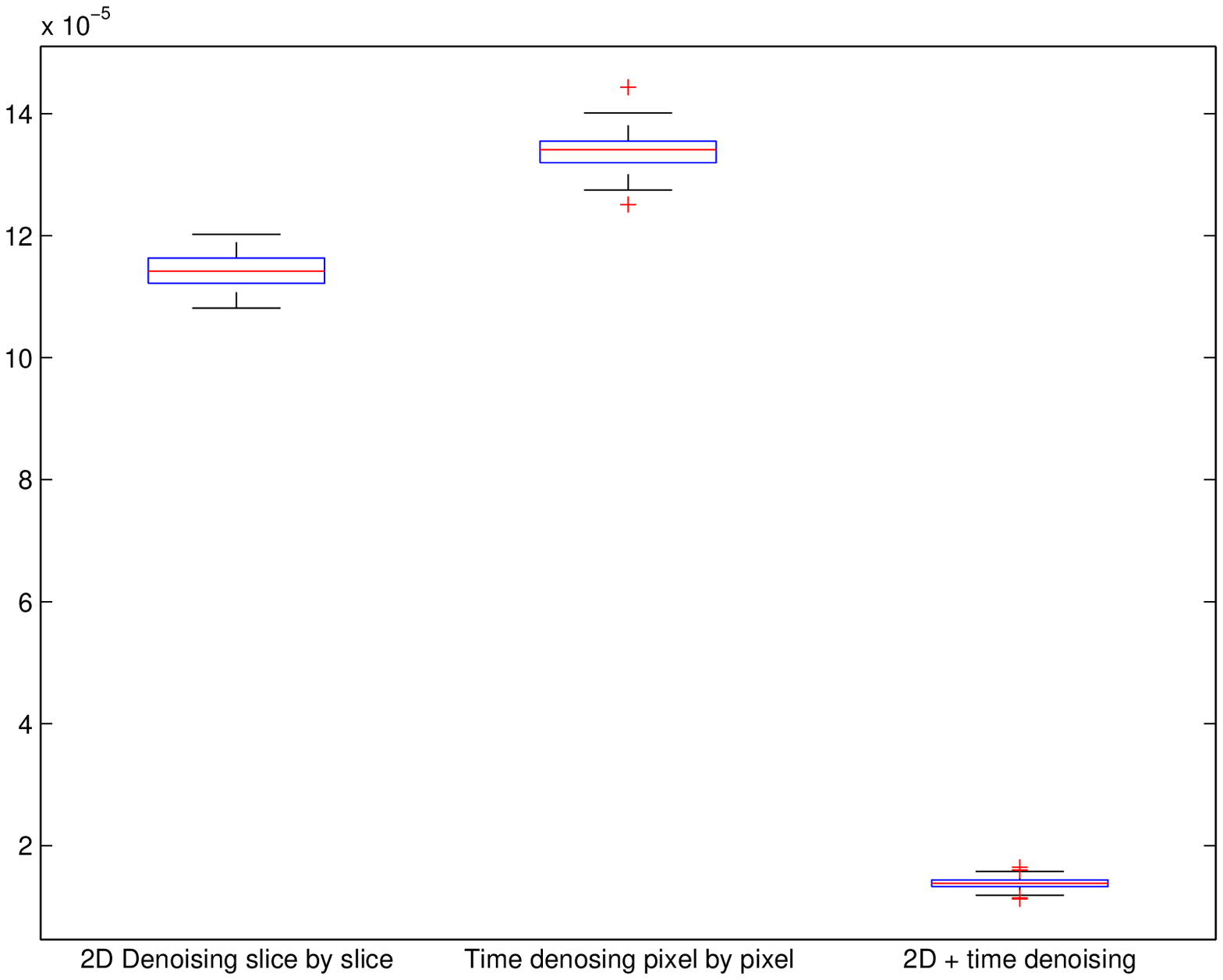} }
\subfigure[]
{ \includegraphics[width=5cm,height=5cm]{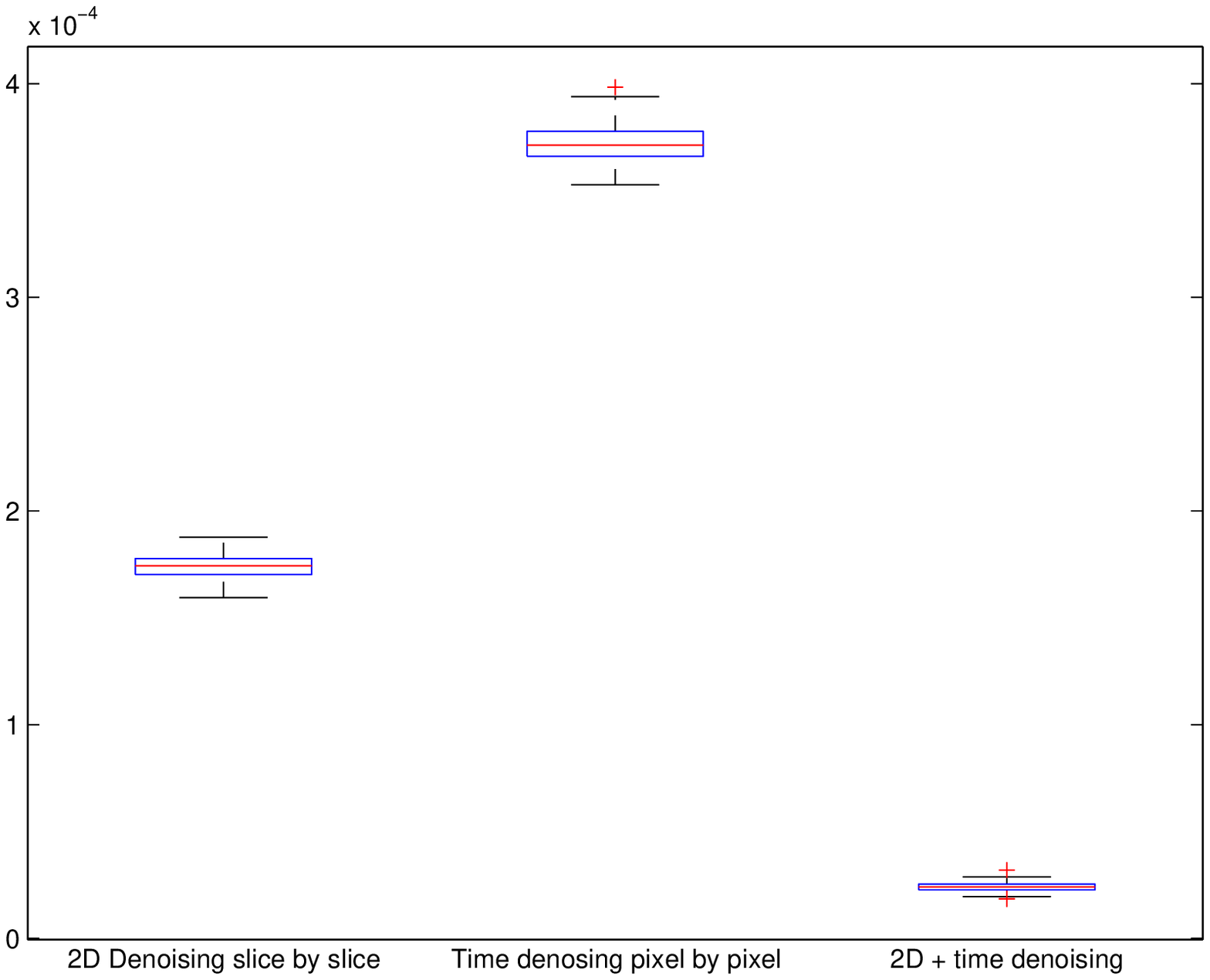} }

\caption{Boxplot of the empirical mean squared \eqref{eq:MSE} error over $M=100$ simulations from model \eqref{eq:modsimus} for the three methods (from left to right: pixel by pixel denoising based on 1D wavelet thresholding, slice by slice denoising based on 2D wavelet thresholding, 2D + time wavelet block thresholding) and for various values of the signal-to-noise ratio (SNR): (a) $SNR = 7$; (b) $SNR = 5$; (c) $SNR = 3$. } \label{fig:MSE}
\end{figure}

\begin{figure}[htbp]
\centering
\subfigure[]
{ \includegraphics[width=7.5cm]{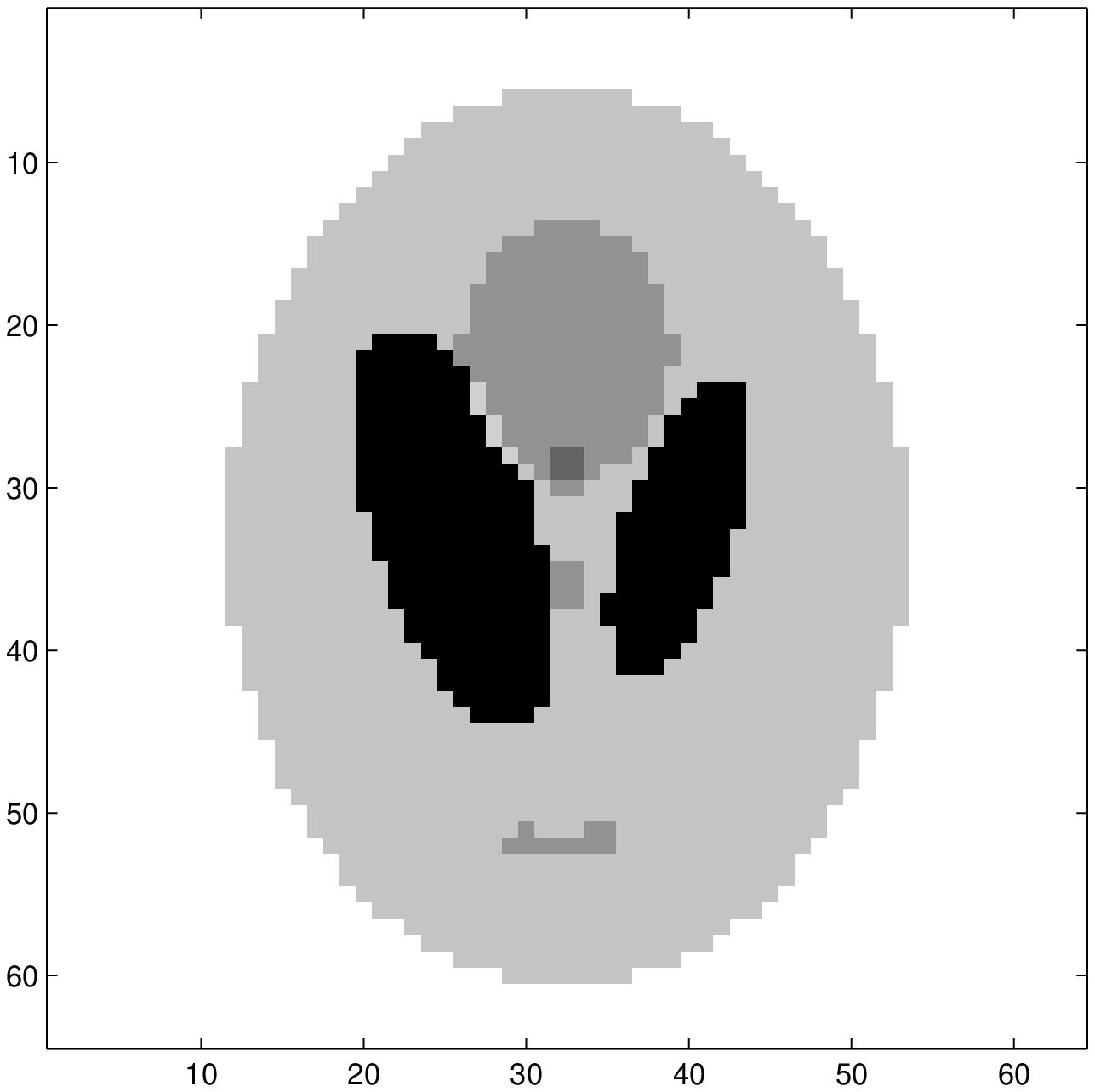}  }
\subfigure[]
{ \includegraphics[width=7.5cm]{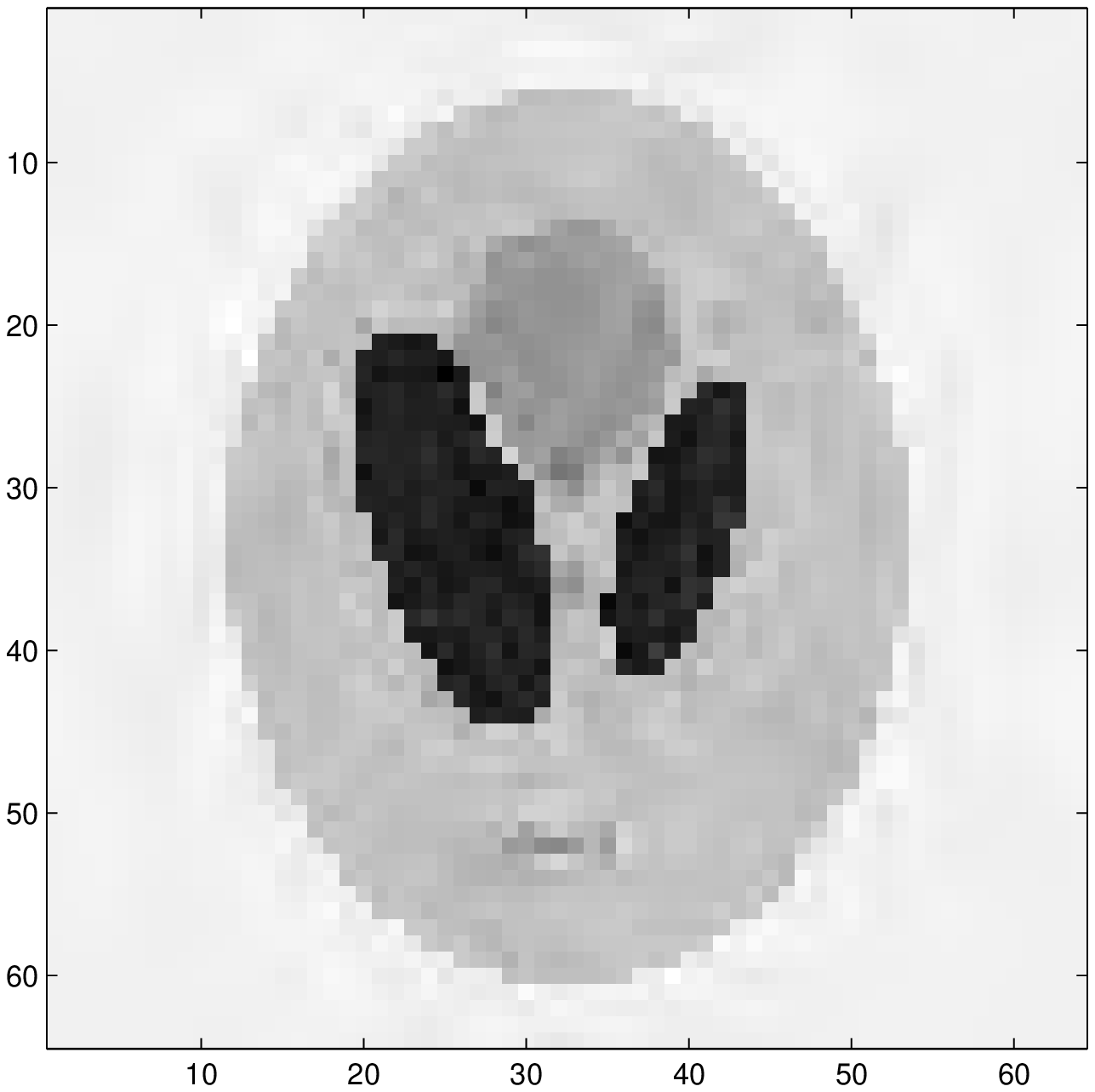} }

\subfigure[]
{ \includegraphics[width=7.5cm]{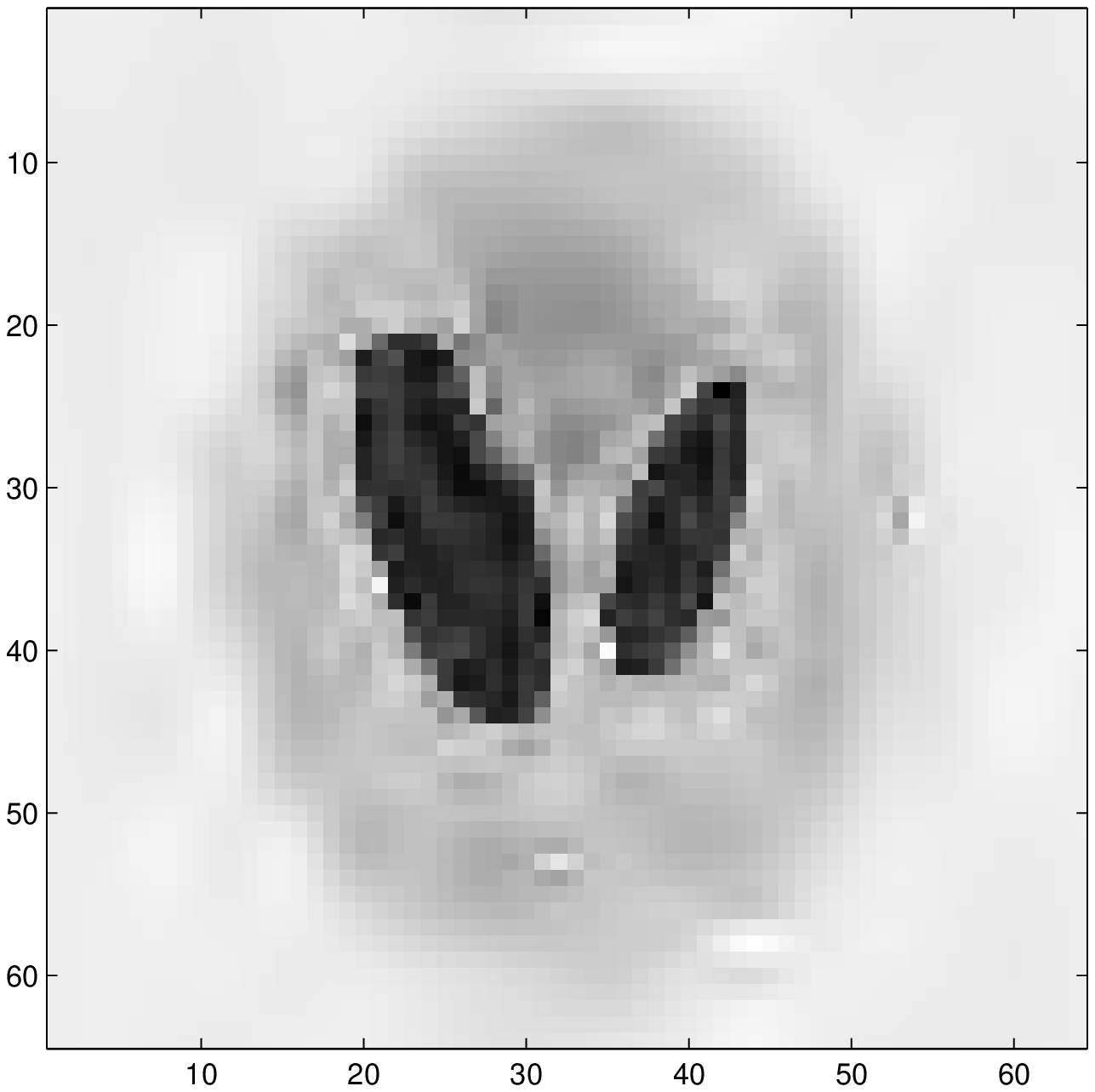}  }
\subfigure[]
{ \includegraphics[width=7.5cm]{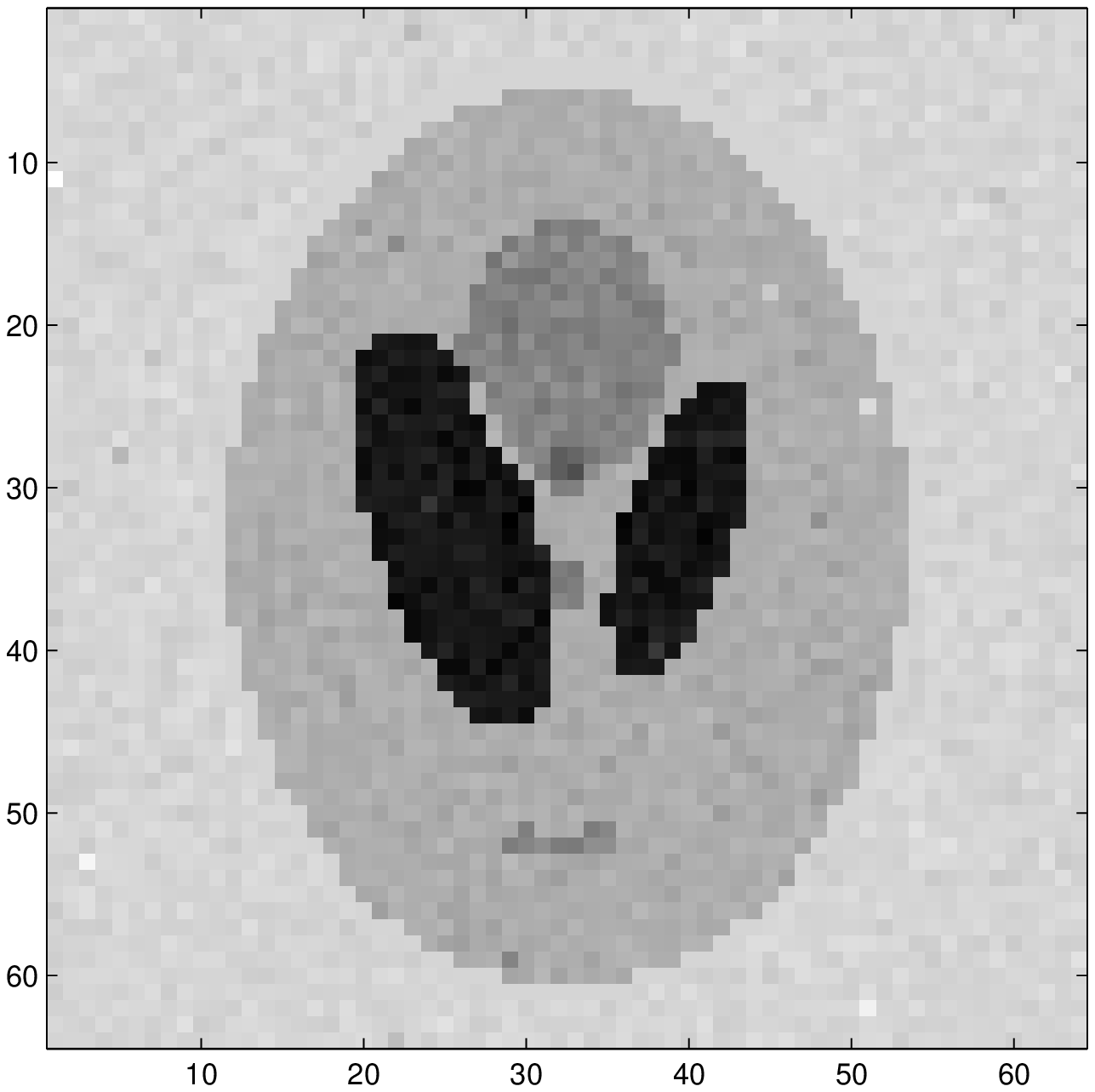} }

\caption{Typical estimates obtained by the three methods with a signal-to-noise ratio $SNR = 5$ for a temporal cut at $t_{\ell} = 0.5748$: (a) true image without additive noise; (b) 2D + time wavelet block thresholding (our method); (c) slice by slice denoising based on 2D wavelet thresholding; (d) pixel by pixel denoising based on 1D wavelet thresholding.} \label{fig:simus_denoising}
\end{figure}

\subsection{Real data}

Now, we return to the real-data example on satellite remote sensing data discussed in Section \ref{sec:motex}. To apply the suggested adaptive nonlinear wavelet estimator, it is necessary to estimate the level of noise in the measurements. For this purpose, we estimate the level of noise in each 2D image at each wavelength using the median absolute deviation (MAD) of the empirical 2D wavelet coefficients at the highest level of resolution (see \cite{ABS01} for further details on this procedure). Then, to apply our method, we took  $\epsilon = \frac{\hat{\sigma}}{\sqrt{n N^2}}$ with $\hat{\sigma}$ being the maximum of these estimated values by MAD over the $n=128$ wavelength, with $N=64$.  The result of our denoising procedure is displayed in Figure \ref{fig:ONERA_denoising}.

\begin{figure}[htbp]
\centering
\subfigure[]
{ \includegraphics[width=7.5cm]{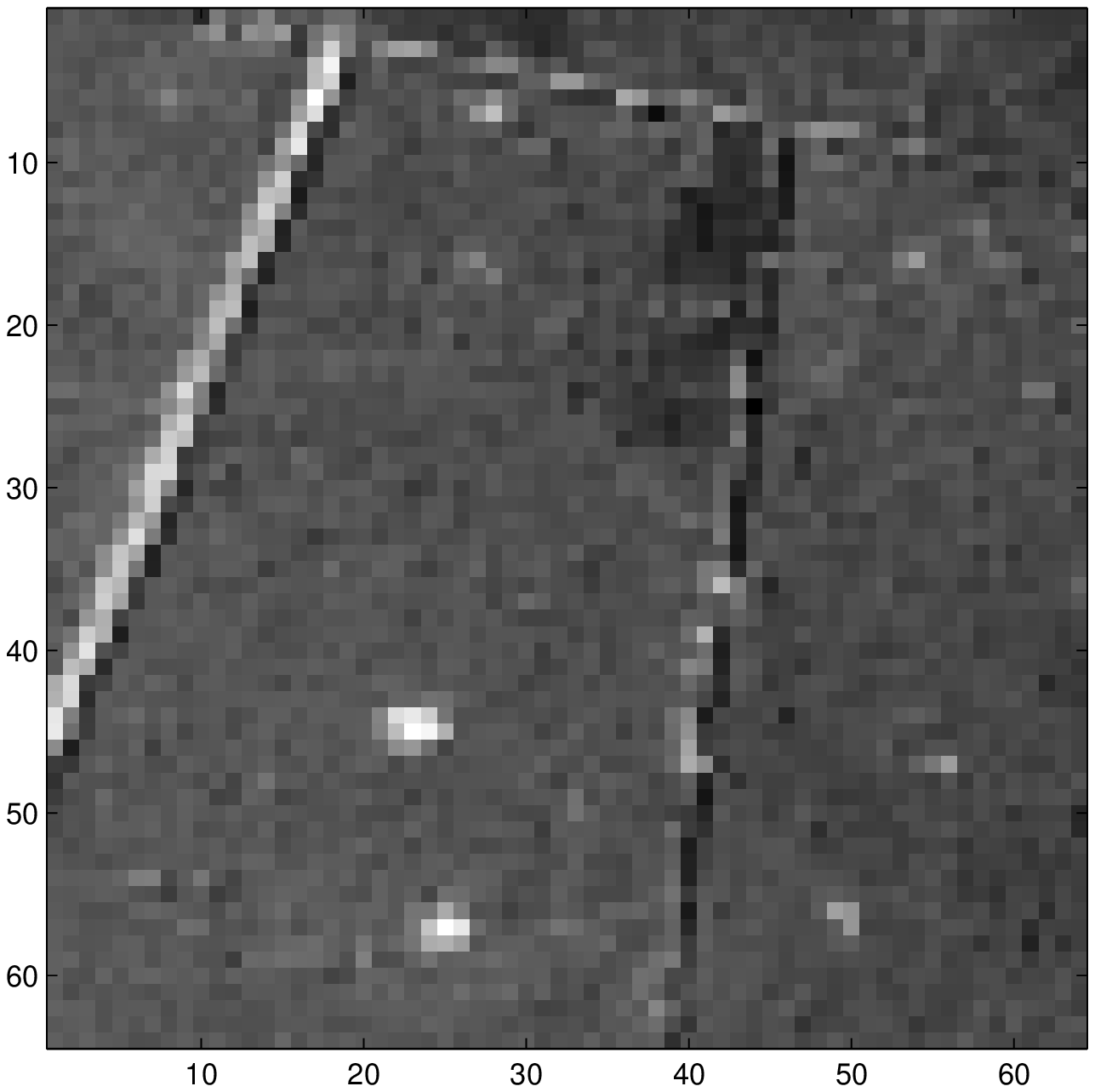}  }
\subfigure[]
{ \includegraphics[width=7.5cm]{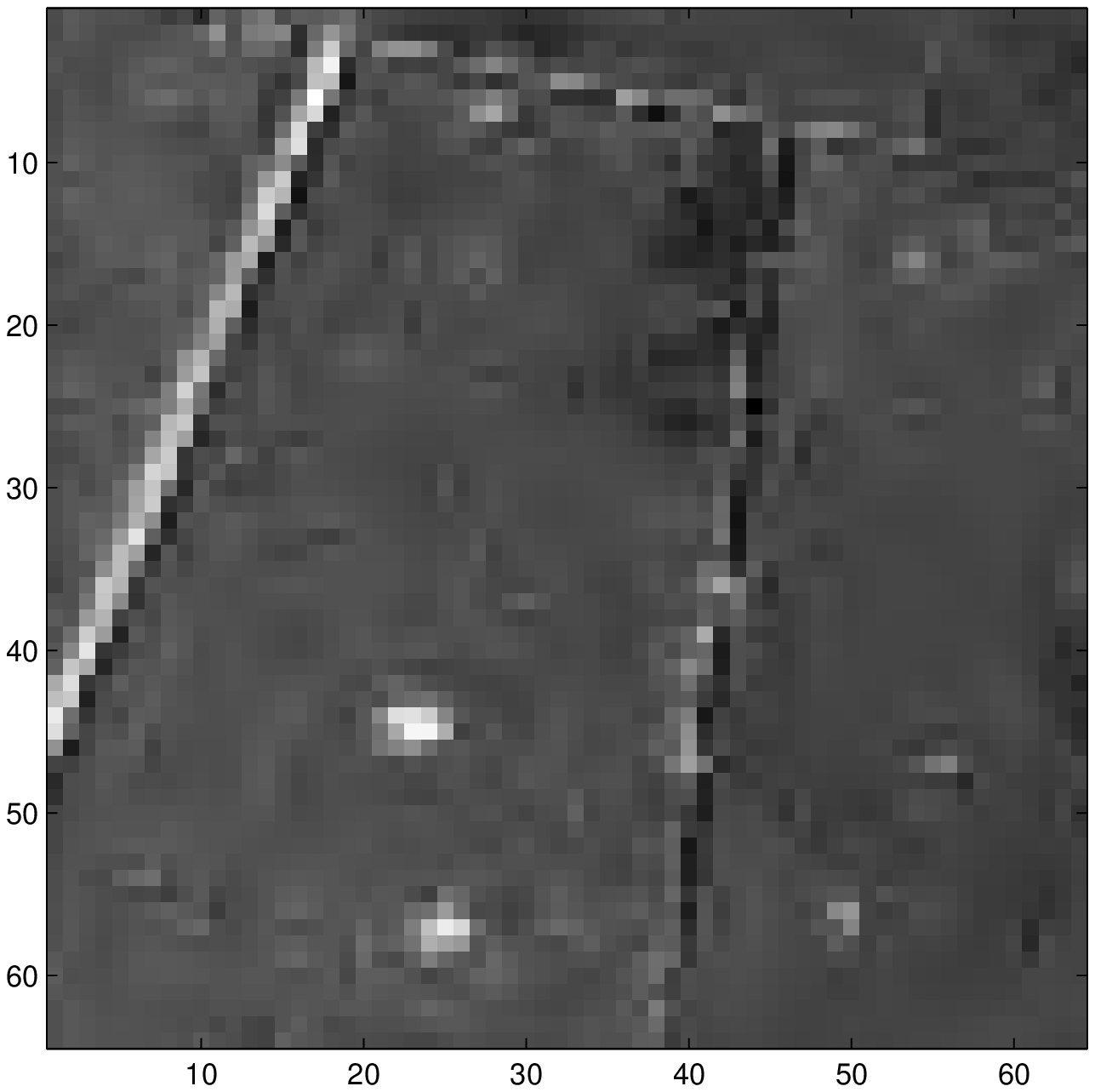} }

\subfigure[]
{ \includegraphics[width=7.5cm]{Figures/Onera_signals.eps}  }
\subfigure[]
{ \includegraphics[width=7.5cm]{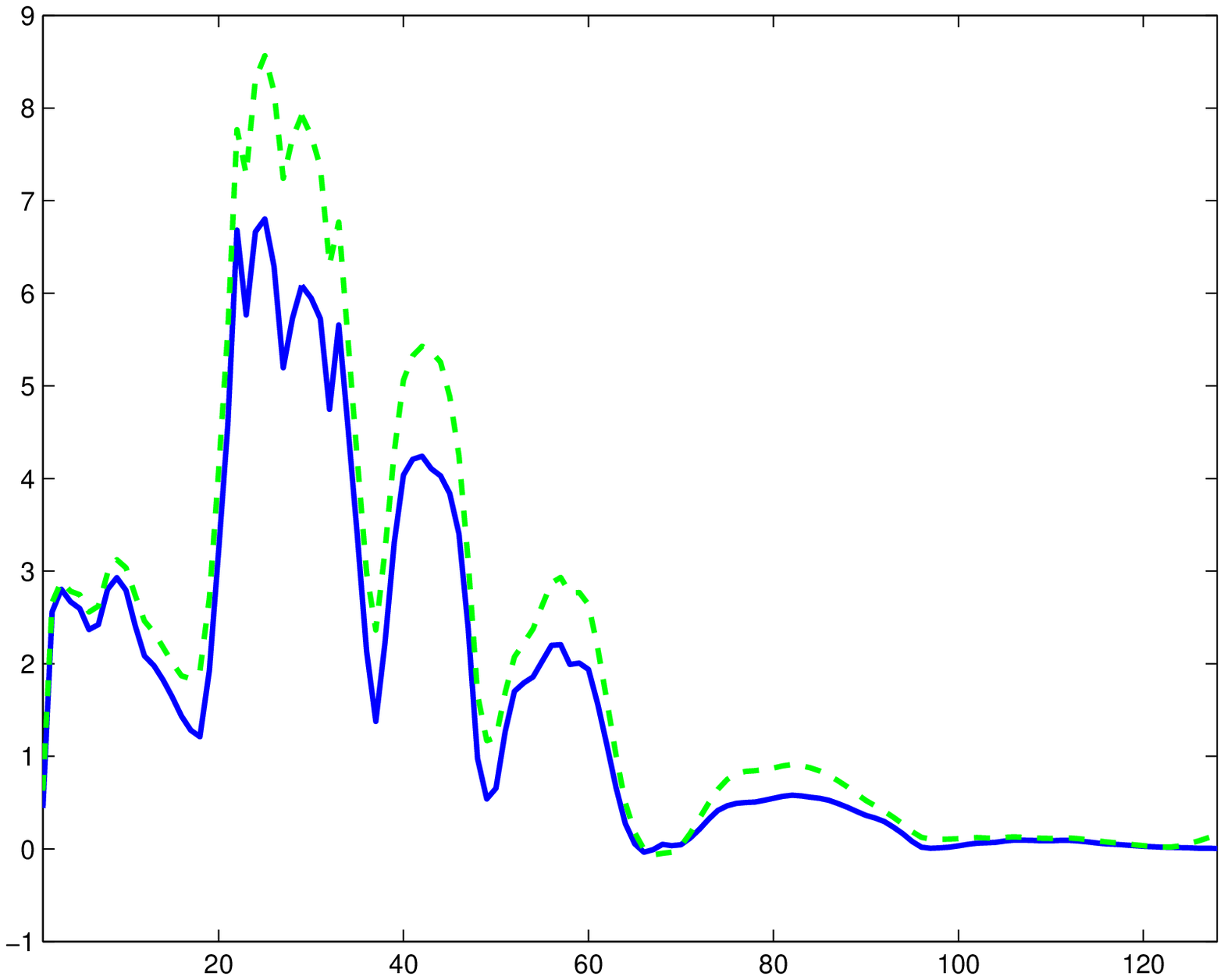} }

\caption{Satellite remote sensing image ($64 \times 64$ pixels, over $128$ wavelengths): (a) a 2D image measured at a specific wavelength (raw data); (b) 2D  image at the same wavelength  obtained after applying  our method; (c) evolution over wavelength of the intensities of the two pixels in green and blue shown in Figure \ref{fig:ONERA}  (raw data); (d) intensities of these two pixels after denoising by our method.} \label{fig:ONERA_denoising}
\end{figure}

\section{Concluding remarks}
\label{conclusions}

We considered the nonparametric estimation problem of time-dependent multivariate functions observed in a presence of additive cylindrical Gaussian white noise of a small intensity. We derived minimax lower bounds for the $L^2$-risk in the proposed spatio-temporal model as the intensity goes to zero, when the underlying unknown response function is assumed to belong to a ball of appropriately constructed inhomogeneous time-dependent multivariate functions. The choice of this class of functions was motivated by real-data examples and  illustrated with the help of an example on satellite remote sensing data. We also proposed both non-adaptive linear and adaptive non-linear wavelet estimators that are asymptotically optimal (in the minimax sense) in a wide range of the so-constructed balls of inhomogeneous time-dependent multivariate functions. The usefulness of the suggested adaptive nonlinear wavelet estimator was illustrated with the help of simulated and real-data examples.

\medskip
Some extensions of the present work are possible. They are briefly mentioned below. 

\medskip
\noindent
[{\tt Inverse Problems}] Model (\ref{model:direct}) can be extended to the case where the signal is observed through a linear operator plus noise. More precisely,  one can consider the nonparametric estimation problem of time-dependent multivariate functions observed through a known or unknown linear operator with kernel $k(\underline{x},\underline{u})$ and in a presence of  additive cylindrical Gaussian white noise,  namely
\begin{equation}
d Y_{\epsilon}(t,\underline{x}) = \left( \int_{\XX} k(\underline{x},\underline{u}) \bfun(t,\underline{u}) d \underline{u} \right) d \underline{x} + \epsilon d W(t,\underline{x}), \label{model:indirect}
\end{equation}
where, as earlier,  $t \in T$ ($T$ is a compact subset of $\RR$) is the time variable, $\underline{x} \in \XX$ ($\XX$ is a compact subset of $\RR^{d}$, $d \geq 1$) is the space variable, $ \bfun \in  \LL^{2}(  T \times \XX )$ is the time-dependent multivariate function that we wish to estimate, $d W(t,\underline{x})$ is a cylindrical orthogonal Gaussian random measure (representing additive noise in the measurements), and $\epsilon >0$ is a small level of noise, that may let be going to zero for studying asymptotic properties. An important example of kernel is the case where 
$$k(\underline{x},\underline{u}) = h(\underline{u}-\underline{x}) \quad \mbox{for some function} \quad h : \RR^2 \to \RR,
$$ 
(with known or unknown singular values) leading to a time-dependent multivariate deconvolution problem. (Note that a sub-class of this model is the case of direct noisy observations of the time-dependent multivariate functions $\bfun(t,\underline{x})$, $t \in T$, $\underline{x} \in \XX$, namely model (\ref{model:direct}) considered in this work.)

\medskip
\noindent
[{\tt Smoothness Assumption}] In either model (\ref{model:direct}) or model (\ref{model:indirect}), instead of using the standard (isotropic) $d$-dimensional Besov spaces on $\XX$ to describe the smoothness of the underlying unknown response function $\bfun(t,\underline{x})$, for each fixed $t \in T$, one could consider anisotropic $d$-dimensional Besov spaces on $\XX$, where different smoothness is assumed in each direction (see, e.g., \cite{MR2111785}). Another possibility, is to consider subclasses of the so-called decomposition spaces that cover both the cases of standard (isotropic) $d$-dimensional Besov spaces as well as, in the case when $d=2$, smoothness spaces corresponding to curvelet-type constructions (see \cite{MR2563260}).

\medskip
The above extensions are projects for future work that we hope to address elsewhere.
 
\section{Appendix}
\label{appF}

\subsection{Besov space and wavelet approximations}

\begin{lemma} \label{lemma:approx}
Let $A_{1} >0$ and $A_{2} >0$ be constants. Let $s_{1} > 0$  and $s_{2} > 0$ be the smoothness parameters in the space and time domains, respectively, such that $0 < s_{1} < \tau_{1}$ and  $0 < s_{2} < \tau_{2}$, where $\tau_1$ and $\tau_2$ are the regularity parameters of the wavelet systems $(\phi,\psi)$ and $(\tilde{\phi},\tilde{\psi})$, respectively. Let $1 \leq p_{1}, q_{1} \leq +\infty$, $1 \leq p_{2}, q_{2}  \leq +\infty$. Assume that $\bfun \in \bB_{p,q}^{s_{1},s_{2}}(A_{1},A_{2})$.  Define $s_{1}' = s_{1}  + d(1/2-1/p_{1}') > 0$ with $p_{1}' = \min(p_{1},2) $ and $s_{2}' = s_{2}  + 1/2-1/p_{2}' > 0$ with $p_{2}' = \min(p_{2},2) $. Let ${\balpha}_{\lambda}(t)$ be defined as in \eqref{eq:balpha}, and $\tilde{\alpha}_{\lambda,m,\ell}$ be defined as in \eqref{eq:alpha}. Then, for every $-1 \leq j < + \infty$,
$$
 \sum_{|\lambda| = j}  \left| {\balpha}_{\lambda}(t) \right|^2 \leq K_{1}' A_{1}^2 2^{-2j s'_{1}},
$$
for some  constant $K_{1}' > 0$, only depending on $s_{1}$ and $p_{1}$, and for every $-1 \leq m < + \infty$,
$$
 \sum_{\ell = 0}^{2^{m}-1} \left|\tilde{\alpha}_{\lambda,m,\ell}\right|^2 \leq K'_{2} A_{\lambda}^2 2^{-2m s'_{2}},
 $$
 for some constant $K'_{2} > 0$, only depending on $s_{2}$ and $p_{2}$.
\end{lemma}
\begin{proof}
Since, for each $t \in T$, $\bfun(t,\cdot) \in B^{s_{1}}_{p_{1},q_{1}}(A_{1})$ with $1 \leq p_{1},q_{1}  \leq +\infty$, using standard embedding properties of Besov spaces, there exists a constant $K_{1}' > 0$, only depending on $s_{1}$ and $p_{1}$, such that  for every $-1 \leq j < + \infty$
$$
 \sum_{|\lambda| = j}  \left| {\balpha}_{\lambda}(t) \right|^2 \leq K_{1}' A_{1}^2 2^{-2j s'_{1}}.
$$
By the definition of $ \bB_{p,q}^{s_{1},s_{2}}(A_{1},A_{2})$, and  using standard embedding properties of Besov spaces, there exists a constant $K'_{2} > 0$, only depending on $s_{2}$ and $p_{2}$, such that  for every $-1 \leq m < + \infty$
$$
\sum_{\ell = 0}^{2^{m}-1} \left|\tilde{\alpha}_{\lambda,m,\ell}\right|^2 \leq K'_{2} A_{\lambda}^2 2^{-2ms'_{2}}.
$$
This completes the proof of the lemma.
\end{proof}

\subsection{A large deviation inequality}

\begin{lemma} \label{lem:deviation}
Let $\delta  > 1$. Then, for any $- 1 \leq m < +\infty$ and $r \in A_{m}$,
\begin{equation}
\label{app:ldr}
\PP \left(  \sum_{\ell \in U_{mr}}   \left|y_{\lambda,m,\ell} - \tilde{\alpha}_{\lambda,m,\ell}\right|^2 \geq  \delta^{2} \epsilon^{2} L_{\epsilon} \right) \leq \epsilon^{\left( \delta -1\right)^2}.
\end{equation}
\end{lemma}

\begin{proof}
The proof is inspired by the arguments used in the proof of Lemma 2 in \cite{MR2488345}.  Consider the set of vectors
$$
\Omega_{m,r} = \left\{ v_{\ell} \in \RR\smallsetminus\{0\}\; : \; \sum_{\ell \in U_{m,r} }  v^2_{\ell} \leq 1 \right\}, 
$$
and the centered Gaussian process defined by
$$
Z_{m,r}(v) =  \sum_{\ell \in U_{m,r} }  v_{m} (y_{\lambda,m,\ell} - \tilde{\alpha}_{\lambda,m,\ell}).
$$
By Lemma 5 in \cite{MR2488345}, we need to find upper bounds for $\EE \left( \sup_{v \in \Omega_{m,r}} Z_{m,r}(v)  \right)$ and $\sup_{v \in \Omega_{m,r}} \var\left( Z_{m,r}(v)  \right)$.  By the Cauchy-Schwartz inequality,
\begin{eqnarray*}
 \sup_{v \in \Omega_{m,r}} Z_{m,r}(v)   &=& \sup_{v \in \Omega_{m,r}}  \sum_{\ell \in U_{m,r} }   v_m (y_{\lambda,m,\ell} - \tilde{\alpha}_{\lambda,m,\ell}) \\
&=&  \left( \sum_{\ell \in U_{m,r} }   (y_{\lambda,m,\ell} - \tilde{\alpha}_{\lambda,m,\ell})^2  \right)^{1/2}.\end{eqnarray*}
Furthermore, Jensen's inequality implies that
\begin{eqnarray*}
\EE \left( \sup_{v \in \Omega_{m,r}} Z_{m,r}(v)  \right) & = & \EE \left( \sum_{\ell \in U_{m,r} }   (y_{\lambda,m,\ell} - \tilde{\alpha}_{\lambda,m,\ell})^2  \right)^{1/2} \\
 & \leq &  \left( \sum_{\ell \in U_{m,r} }   \EE (y_{\lambda,m,\ell} - \tilde{\alpha}_{\lambda,m,\ell})^2  \right)^{1/2} = \epsilon  L_{\epsilon}^{1/2}   .
\end{eqnarray*}
By independence of $y_{\lambda,m,\ell} - \tilde{\alpha}_{\lambda,m,\ell}$ and $y_{\lambda,m,\ell'} - \tilde{\alpha}_{\lambda,m,\ell'}$ for $\ell \neq \ell'$, we obtain
\begin{eqnarray*}
\sup_{v \in \Omega_{m,r}} \var\left( Z_{m,r}(v)  \right) & = & \sup_{v \in \Omega_{m,r}}  \sum_{\ell \in U_{m,r} }  v_{\ell}^{2}  \var (y_{\lambda,m,\ell} - \tilde{\alpha}_{\lambda,m,\ell})^2  \leq \epsilon^{2} .
\end{eqnarray*}
Thus, by Lemma 5 in \cite{MR2488345}, one has that
$$
\PP \left(   \left( \sum_{\ell \in U_{m,r} }   (y_{\lambda,m,\ell} - \tilde{\alpha}_{\lambda,m,\ell})^2  \right)^{1/2} \geq x + \epsilon  L_{\epsilon}^{1/2} \right) \leq \exp\left(-\frac{x^{2}}{2 \epsilon^2}\right),
$$
for any $x > 0$. Finally, by taking $x = (\delta-1) \epsilon  L_{\epsilon}^{1/2}$, we arrive at \eqref{app:ldr}, thus completing the proof of the lemma.
\end{proof}

\section*{Acknowledgements}
J\'er\'emie Bigot is grateful for the hospitality of the Department of Mathematics and Statistics at the
University of Cyprus, Cyprus, and Theofanis Sapatinas is grateful
for the hospitality of DMIA/ISAE at Toulouse University, France, where parts of this
work were carried out. The authors would like to acknowledge helpful discussions with G\'erard Kerkyacharian, Marianna Pensky and Dominique Picard at an early stage of this work. 

\bibliographystyle{apalike}
\bibliography{time_dependent_denoising}

 \end{document}